\documentclass{article}    

\usepackage{amsmath}        
\usepackage{amsfonts}       
\usepackage{graphicx}       
\usepackage{xcolor}          
\usepackage{subfig}         
\usepackage{booktabs}	
\usepackage{fullpage}
\usepackage{booktabs}
\usepackage{amsthm}

\usepackage{multirow}
\usepackage{fullpage}

\newcommand{\R}{\mathbb{R}}
\newcommand{\N}{\mathbb{N}}
\newcommand{\dd}{\mathrm{d}}
\newcommand{\kC}{\mathcal{C}}
\newcommand{\kI}{\mathcal{I}}
\newcommand{\kJ}{\mathcal{J}}
\newcommand{\kT}{\mathcal{T}}
\newcommand{\kF}{\mathcal{F}}

\newcommand{\f}{{f}}

\newcommand{\T}[1]{#1^T}

\newcommand{\abs}[1]{\left|{#1}\right|}

\newcommand{\uvoz}[1]{``#1"}

\newtheorem{theorem}{Theorem}
\newtheorem{lemma}{Lemma}
\newtheorem{corollary}[lemma]{Corollary}
\newtheorem{remark}{Remark}
\newtheorem{definition}[lemma]{Definition}

\begin{document}

\title{Godunov--like numerical fluxes for conservation laws on networks\thanks{The work of L. Vacek is supported by the Charles University, project GA UK No. 1114119. The work of V. Ku\v{c}era is supported by the Czech Science Foundation, project No. 20-01074S.}}

\author{%
{\sc
Luk\'a\v s Vacek\thanks{Corresponding author. Email: lvacek@karlin.mff.cuni.cz}} \\[2pt]
Charles University, Faculty of Mathematics and Physics\\
Sokolovsk\'{a} 83, Praha 8, 186\,75, Czech Republic\\[6pt]
{\sc and}\\[6pt]
{\sc V\'aclav Ku\v cera}\thanks{Email: kucera@karlin.mff.cuni.cz}\\[2pt]
Charles University, Faculty of Mathematics and Physics\\
Sokolovsk\'{a} 83, Praha 8, 186\,75, Czech Republic
}

\maketitle

\begin{abstract}
This paper deals with the construction of a discontinuous Galerkin scheme for the solution of Lighthill-Whitham-Richards traffic flows on networks. The focus of the paper is the construction of two new numerical fluxes at junctions, which are based on the Godunov numerical flux. We analyze the basic properties of the two Godunov-based fluxes and the resulting scheme, namely conservativity and the traffic distribution property. We prove that if the junction is not congested, the traffic flows according to predetermined preferences of the drivers. Otherwise a small traffic distribution error is present, which we interpret as either the existence of dedicated turning lanes, or factoring of human behavior into the	model. We compare our approach to that of \v Cani\'c et al. (J. Sci. Comput., 2015). Numerical experiments are provided.
\end{abstract}

\section{Introduction}
In this paper, we are concerned with the simulation of the movement of traffic on \emph{networks} of roads. We take the \emph{macroscopic} approach, where the traffic is modeled as a uniform continuum which moves through the roads. This is opposed to the microscopic approach, where each individual vehicle is modeled separately. Since the total number of vehicles is conserved, the basic mathematical model for us will be that of partial differential equations (PDEs) describing conservation laws, namely nonlinear first order hyperbolic PDEs. Specifically, we are concerned with the so-called Lighthill-Whitham-Richards (LWR) traffic flow model, where the traffic moves according to an equilibrium flow of homogeneous traffic, which is described by a so-called \emph{fundamental diagram}, cf. \cite{Greenshields35}, \cite{Lecture_notes_Jungel}, \cite{Mathematical_Framework} or \cite{vanWageningen-Kessels} for an overview. Such approaches to modelling traffic flows on a single road are more or less standard. What is considerably newer and less studied is the generalization of the LWR models to networks of roads, which can be described by an \emph{oriented graph}, where on each road we have the equations for the LWR model and we need to supply some kind of boundary condition at intersections which correspond to vertices of the graph cf. \cite{Networks}.

In our case, we use the \emph{discontinuous Galerkin} (DG) method to discretize the LWR model on each individual road and the behavior of traffic at junctions is determined by prescribing a \emph{numerical flux} at the junction. In such a case, one must take into account not only the necessity for the resulting scheme to be conservative (vehicles are not lost or formed at intersections), but also other properties, such as taking into account the preferences of individual drivers as to which outgoing road they wish to take from the junction, etc. Such numerical fluxes were constructed and applied in \cite{Networks} and \cite{RKDG}, where it is assumed that the drivers behave in such a way as to maximize the total flux through the intersection. This leads to a complicated linear programming problem, which can be explicitly solved (giving an explicit construction of the numerical flux) only in the simplest cases, cf. \cite{Networks} and \cite{RKDG}. We take a slightly different approach, where the resulting numerical fluxes are explicitly constructed on an arbitrary junction based on the traffic distribution requirements. The construction seems more natural for human drivers, who are mainly concerned with the traffic density at their specific pair of incoming and outgoing roads and are (somewhat selfishly) not concerned with maximizing the total flux of all the traffic through the junction. The latter case is much more realistic e.g. for a swarm of centrally coordinated or communicating autonomous vehicles. 

We have already described the basic construction of our numerical flux in the paper \cite{Prvni_clanek} which was however based on a generic classical numerical flux, as used in DG methods on single roads. In this paper, we refine the construction and base it on the \emph{Godunov} numerical flux, which is based on the exact solution of a local Riemann problem, and is therefore a  natural numerical flux also from the point of view of the PDE theory, since it corresponds to the so called Bardos-LeRoux-Nédélec boundary conditions, which is the correct way how to prescribed Dirichlet data on a boundary. In this paper we take the classical Godunov numerical flux and use it to construct a Godunov-like numerical flux at a general junction, which is based on the drivers' preferences described by \emph{traffic distribution coefficients}. Actually, we derived two similar numerical fluxes that differ in the way the traffic distribution coefficients are treated. We then proceed to analyze the resulting DG scheme, namely we prove a discrete analogue of the Rankine-Hugoniot condition at the junction, from which we derive global conservativity of the scheme across the whole network. As it turns out, our numerical flux(es) do not exactly satisfy in all situations the apriori preferences of the drivers in the form of relations given by the traffic distribution coefficients. We analyze this effect in detail and derive relations for the \emph{traffic distribution error}, which we then interpret. Namely, we can show that if (in some sense) a junction is not congested, the drivers follow their predetermined preferences. However this is not true when there is a congestion in the junction, and the original traffic distribution coefficients are not satisfied exactly. We interpret this in two ways -- (a) as typical human behavior, where some drivers decide to change their route if they see that the one they originally chose is blocked, or (b) that there are dedicated turning lanes in the incoming roads before the junctions. Such turning lanes allow some drivers, whose route of choice is not blocked to pass the ones whose preferred outgoing road is congested. If there are no dedicated turning lanes than a traffic jam on one of the outgoing roads blocks all of the traffic in the entire junction, even though other outgoing roads may be completely free. This is what happens in the numerical flux considered in \cite{Networks}, \cite{RKDG}, but not in our case. Finally, we demonstrate the mentioned phenomena on simple numerical test cases. 

The structure of the paper is as follows. In Section \ref{sec:1} we introduce the basic concepts and notation needed to describe traffic flows on networks along with the traffic distribution coefficients describing the drivers' preferences. In Section \ref{sec:2}, we introduce the discontinuous Galerkin method on single roads and introduce and reformulate the classical Godunov numerical flux. In Section \ref{sec_numerical_fluxes_junctions} we review the approach of \cite{Networks}, \cite{RKDG} and construct our Godunov-like numerical fluxes at junctions along with the corresponding DG formulation on networks. In Section \ref{sec:5}, we prove the conservativity of the DG scheme on networks with the considered Godunov fluxes and analyze the traffic distribution error. Finally, we present numerical results in Section \ref{sec_results}.

\section{Macroscopic traffic flow models on networks}
\label{sec:1}
We begin with the mathematical description of macroscopic vehicular traffic, cf. \cite{Lecture_notes_Jungel}, \cite{Mathematical_Framework} and \cite{vanWageningen-Kessels} for a more detailed treatment of the subject. First, we consider a single road described mathematically as a one-dimensional interval. In basic macroscopic models, traffic flow is described by three fundamental quantities -- \emph{traffic flow $Q$, traffic density $\rho$} and \emph{mean traffic flow velocity $V$}, all of these being functions of both the spatial position $x$ and time $t$.

The basic governing equation of traffic flow is derived using the assumption that the number of cars in an arbitrary segment $[x_1,x_2]$ of the road changes only due to the flux through the endpoints, i.e. 
\begin{equation}
\dfrac{\dd}{\dd t}\int_{x_1}^{x_2}\rho(x,t)\,\dd x = Q(x_1,t)-Q(x_2,t).
\end{equation}
Writing the right-hand side as an integral and eliminating the integral over the arbitrarily chosen $[x_1,x_2]$ gives the conservation law for $\rho$ in the form
\begin{equation}\label{Conservation_law}
\dfrac{\partial}{\partial t}\rho(x,t)+\dfrac{\partial}{\partial x}Q(x,t)=0.
\end{equation}
Equation (\ref{Conservation_law}) must be supplemented by an initial condition
and appropriate boundary conditions which we will treat in detail in the case of networks of roads.

Equation (\ref{Conservation_law}) is underdetermined, as there is a single equation for two unknowns. Thus we need to supply another equation or relation between the variables. Greenshields described a relation between traffic density and traffic flow in the paper \cite{Greenshields35}. He made the assumption derived from observations that in homogeneous traffic (traffic with no changes in time and space), the traffic flow $Q$ is a function which depends only on the density $\rho$. Let us denote the equilibrium flow of homogeneous traffic as $Q_e$, derived from $Q$. The relationship between the $\rho$ and $Q_e$ is described by the so-called \emph{fundamental diagram}. The approach where we use the equilibrium traffic flow $Q_e$ in equation (\ref{Conservation_law}) is called the Lighthill-Whitham-Richards (LWR) traffic flow model and results in the equation
\begin{equation}\label{LWR_problem}
\begin{split}
\rho_t+\big(Q_e(\rho)\big)_x=0,&\qquad x\in\R,\ t>0,\\
\rho(x,0)=\rho_0(x),&\qquad x\in\R.
\end{split}
\end{equation}
Equation (\ref{LWR_problem}) belongs to the class of \emph{nonlinear first order hyperbolic equations} and, for practical purposes will be considered on finite intervals with appropriate boundary conditions. 

There are many different proposals for the equilibrium traffic flow $Q_e$ derived from real traffic data, cf.  \cite{Mathematical_Framework}. Here we present only \emph{Greenshields model}, which defines the equilibrium traffic flow as
\[
Q_e(\rho)=v_{\max}\,\rho\left( 1-\dfrac{\rho}{\rho_{\max}}\right),
\]
where $v_{\max}$ is the maximal velocity and $\rho_{\max}$ is the maximal density. We can see the fundamental diagram in Figure \ref{obr12}, where $v_{\max}=\rho_{\max}=1$.
\begin{figure}[b!]\centering
\subfloat[Velocity--density diagram.]{\includegraphics[height=1.5in]{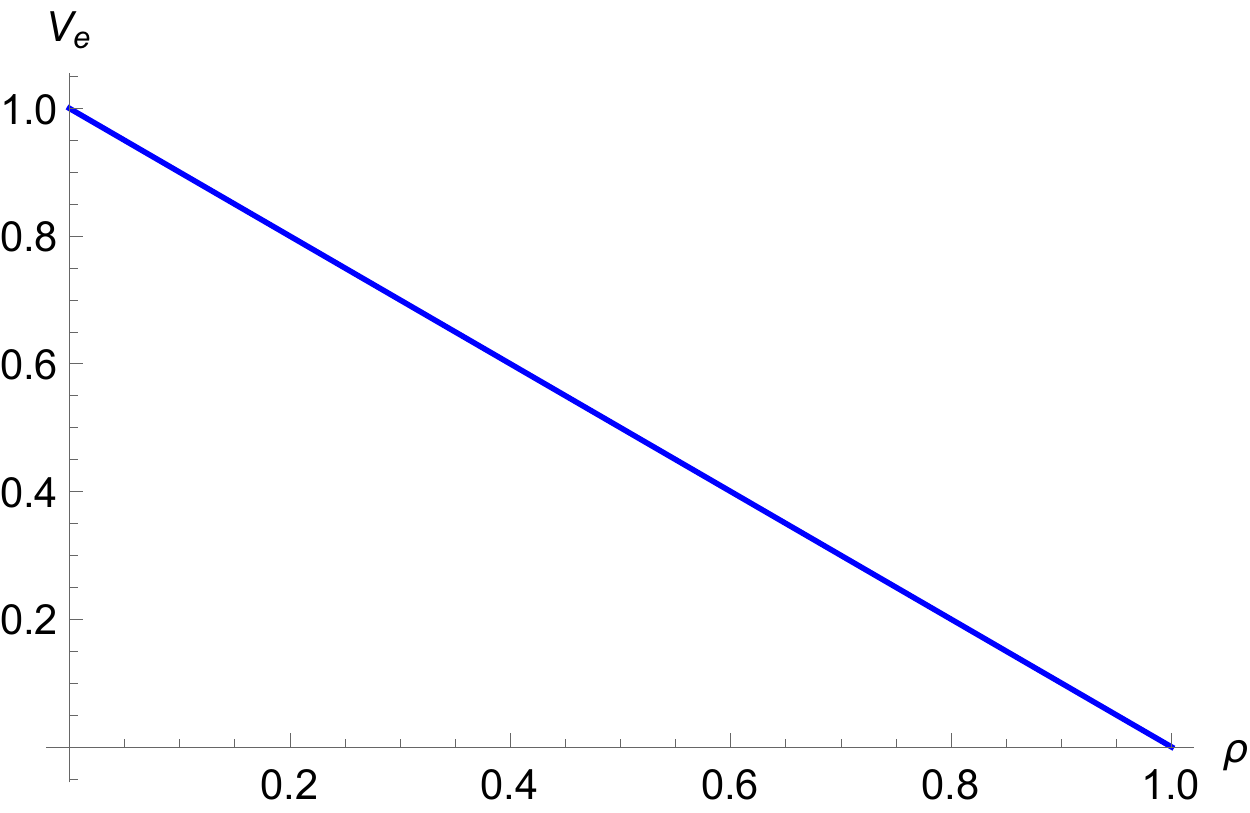}}
\hspace{5pt}
\subfloat[Flow--density diagram.]{\includegraphics[height=1.5in]{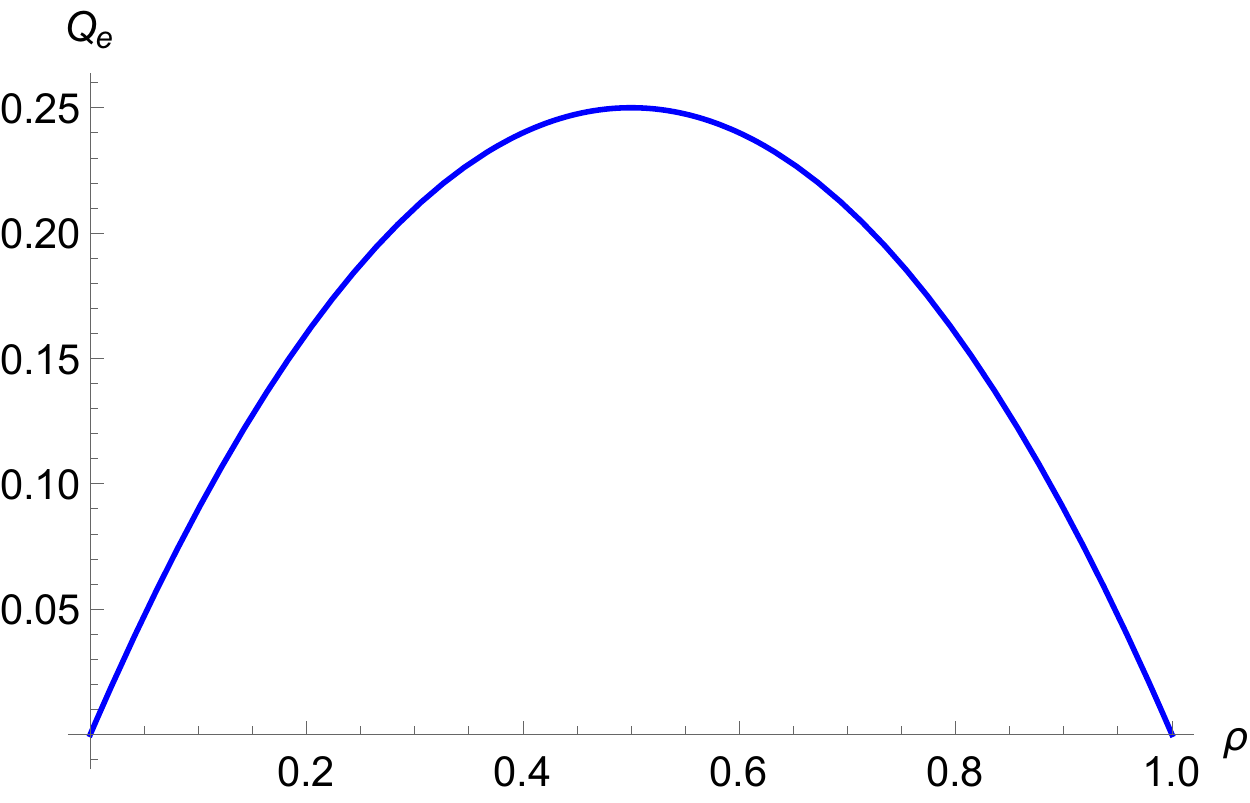}}
\caption{Fundamental diagrams of the Greenshields model.}
\label{obr12}
\end{figure}

Now we consider a road network represented by a directed graph, following \cite{Networks}. The graph is a finite collection of directed edges (roads), connected together at vertices (intersections). Each vertex has a finite set of \emph{incoming edges} and \emph{outgoing edges}. 
\begin{definition}[Network]\label{Def_network}	
We define a \emph{network} as a couple $(\kI,\kJ)$, where $\kI=\lbrace I_n\rbrace_{n=1}^N$ is a finite set of edges and $\kJ=\lbrace J_m\rbrace_{m=1}^M$ is a finite set of vertices. Each edge $I_n$ is represented by an interval $[a_n,b_n]\subseteq[-\infty,\infty],\ n=1,\ldots,N$. Each vertex $J_m$ is a union of two non--empty subsets $\text{Inc}(J_m)$ and $\text{Out}(J_m)$ of $\lbrace 1,\ldots,N\rbrace$ representing \emph{incoming} and \emph{outgoing edges}, respectively. We assume the following:
\begin{itemize}
\item[(i)] For all $J_i,J_j\in\kJ,\ i\neq j:\text{Inc}(J_i)\cap\text{Inc}(J_j)=\emptyset$ and $\text{Out}(J_i)\cap\text{Out}(J_j)=\emptyset$.
\item[(ii)] If $i\notin\cup_{J\in\kJ}\text{Inc}(J)$, $i\in\lbrace 1,\ldots,N\rbrace$, then $b_i=\infty$ and if $i\notin\cup_{J\in\kJ}\text{Out}(J)$, $i\in\lbrace 1,\ldots,N\rbrace$, then $a_i=-\infty$. Moreover, for all $i\in\lbrace 1,\ldots,N\rbrace: i\in\cup_{J\in\kJ}\text{Inc}(J)$ or $i\in\cup_{J\in\kJ}\text{Out}(J)$.
\end{itemize}
\end{definition}

Condition (i) states that each edge can be incoming for at most one vertex and outgoing for at most one vertex. Condition (ii) states that edges that are connected to only one vertex extend to $\pm\infty$. Of course in practice artificial inflow/outflow boundaries are introduced on such edges in the numerical solution. We can see an example in Figure \ref{obr15}.

\begin{figure}[t!]\centering
\includegraphics[height=1.5in]{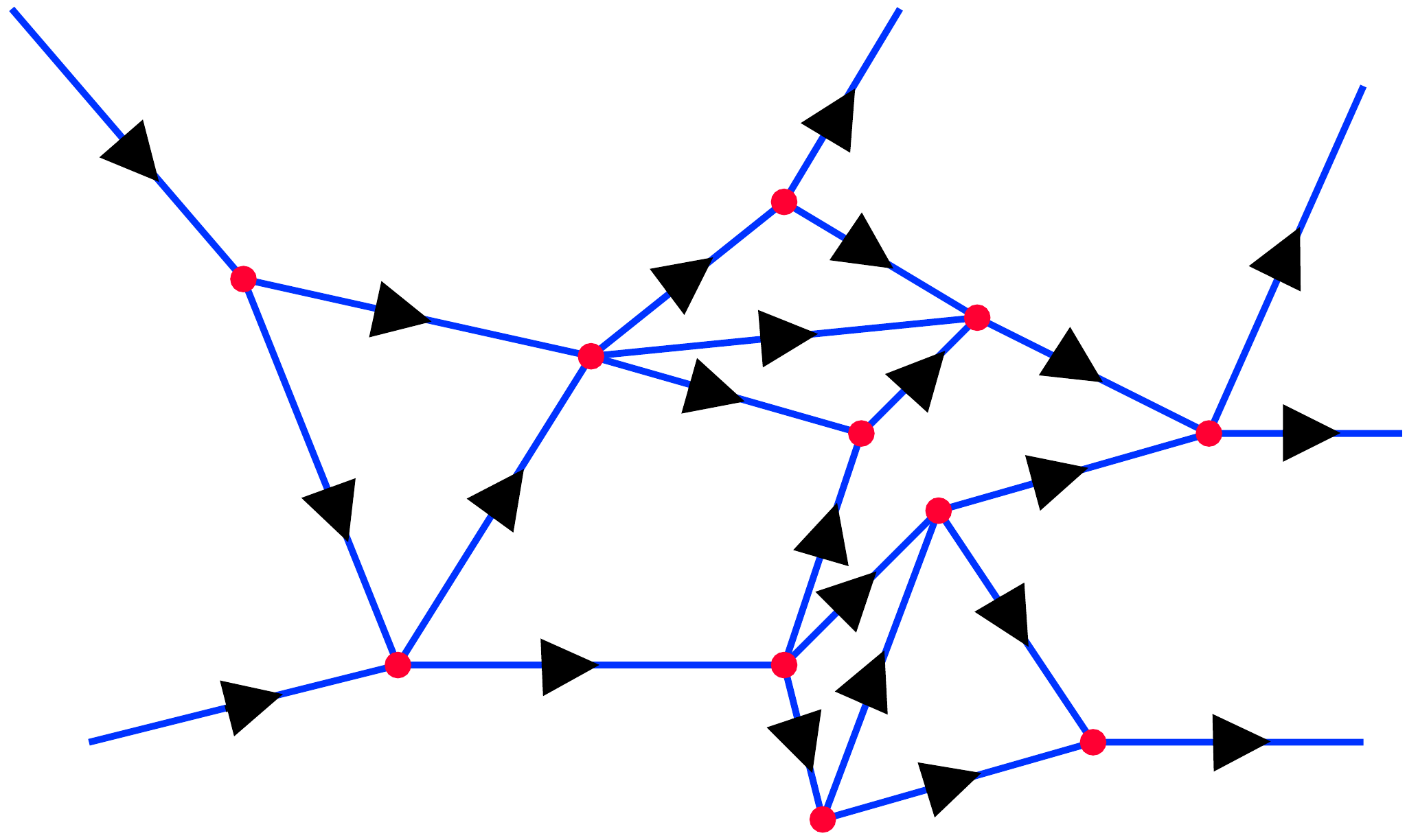}
\caption{Example of a network.}
\label{obr15}
\end{figure}

As we are dealing with traffic flows described by LWR models, we assume that the traffic on each edge number $i\in\lbrace 1,\ldots,N\rbrace$ is described by a conservation law of the general form
\begin{equation}\label{LWR_network}
\begin{split}
({u}_i)_t+\left(\f({u}_i)\right)_x=0,&\qquad x\in(a_i,b_i),\ t>0,\\
{u}_i(x,0)={u}_{0,i}(x),&\qquad x\in(a_i,b_i),
\end{split}
\end{equation}
where ${u}_i:(a_i,b_i)\times[0,\infty)\to\R$ is the traffic density on the $i$-th edge (road). We note that although our primary interest are traffic flow models, the system of equations (\ref{LWR_network}) can represent general conservation laws, which is the reason why we have abandoned the notation based on traffic density $\rho$ and traffic flow $Q$ and use the generic notation $u$ and $f$ in the governing equations. We note that we will assume that the convective flux $f$ has a global maximum at some \emph{critical value} (\emph{critical density}) $u_*$ and $f$ is non--decreasing on the interval $(-\infty,u_*]$ and non--increasing on $[u_*,\infty)$, cf. Section \ref{Subsec:Godunov}. This assumption is typical for traffic flow models.

What remains is to describe the behavior of traffic at junctions. For this purpose it is sufficient to first consider a single vertex (junction) and its incoming and outgoing roads for simplicity. The resulting considerations can then be applied separately to each vertex of the general network. 

We consider a network $(\kI,\kJ)$ and fix a vertex $J\in\kJ$ for which we assume that $\text{Inc}(J)=\lbrace 1,\ldots,n\rbrace$ and $\text{Out}(J)=\lbrace n+1,\ldots,n+m\rbrace$. We define the spatial limits of traffic densities on individual roads at the common vertex $J$ as
\[
{u}_i^{-}(b_i,t):=\lim_{x\to b_i\!-}{u}_i(x,t)\quad\text{and}\quad {u}_j^{+}(a_j,t):=\lim_{x\to a_j\!+}{u}_j(x,t)
\]
for all $i=1,\ldots,n$ and $j=n+1,\ldots,n+m$. Now we can define the solution at a junction.
\begin{definition}[Traffic solution at a junction]\label{def_Traffic_solution_at_junction}
Let $J$ be a junction with incoming roads $I_1,\ldots,I_n$ and outgoing road $I_{n+1},\ldots,I_{n+m}$. Then we define a weak solution at $J$ as a collection of functions ${u}_l:I_l\times[0,\infty)\rightarrow\R$, $l=1,\ldots,n+m$ such that
\[
\sum_{l=1}^{n+m}\int_{a_l}^{b_l}\int_0^\infty\left( {u}_l\dfrac{\partial\varphi_l}{\partial t}+\f({u}_l)\dfrac{\partial\varphi_l}{\partial x}\right)\dd t\,\dd x=0
\]
holds for every $\varphi_l\in\kC_0^1([a_l,b_l]\times [0,\infty))$, $l=1,\ldots,n+m$, that are also smooth across the junction, i.e.
\[
\varphi_i^{-}(b_i,\cdot)=\varphi_j^{+}(a_j,\cdot),\qquad\left(\dfrac{\partial\varphi_i}{\partial x}\right)^{-}(b_i,\cdot)=\left(\dfrac{\partial\varphi_j}{\partial x}\right)^{+}(a_j,\cdot),
\]
for all $i\in\lbrace 1,\ldots,n\rbrace$ and $j\in\lbrace n+1,\ldots,n+m\rbrace$.
\end{definition}

The basic property of the weak solution from Definition \ref{def_Traffic_solution_at_junction} is that it satisfies the Rankine-Hugoniot condition which is essentially the conservation of vehicles at the junction.
\begin{lemma}
Let ${u}=\T{({u}_1,\ldots,{u}_{n+m})}$ be a weak solution at the junction $J$ such that each $x\rightarrow{u}_i(x,t)$ has bounded variation. Then ${u}$ satisfies the \emph{Rankine-Hugoniot condition}
\begin{equation}\label{Junction_equation}
\sum_{i=1}^n \f({u}_i^{-}(b_i,t))=\sum_{j=n+1}^{n+m} \f({u}_j^{+}(a_j,t))
\end{equation}
for almost every $t>0$ at the junction $J$.
\end{lemma}
\begin{proof}
The proof is a simple application of integration by parts and can by found in  \cite[Lemma 5.1.9]{Networks}.
\end{proof}

Definition \ref{def_Traffic_solution_at_junction} simply enforces the conservation of vehicles at $J$. However it is also necessary to take into account the preferences of drivers how the traffic from incoming roads is distributed to outgoing roads according to some predetermined coefficients.

\begin{definition}[Traffic distribution matrix]
\label{def:Traffic_distribution_matrix}
Let $J$ be a fixed vertex with $n$ incoming edges and $m$ outgoing edges. We define a \emph{traffic distribution matrix} $A$ as
\[
A=\begin{bmatrix}
\alpha_{n+1,1} & \cdots & \alpha_{n+1,n}\\
\vdots & \vdots & \vdots \\
\alpha_{n+m,1} & \cdots & \alpha_{n+m,n}
\end{bmatrix},
\]
where $0\leq\alpha_{j,i}\leq 1$ for all $i\in\lbrace 1,\ldots,n\rbrace$, $j\in\lbrace n+1,\ldots,n+m\rbrace$ and
\begin{equation}\label{sum_alpha}
\sum_{j=n+1}^{n+m}\alpha_{j,i}=1
\end{equation}
holds for all $i\in\lbrace 1,\ldots,n\rbrace$.
\end{definition}

The $i^{th}$ column of $A$ describes how the traffic from the incoming road $I_i$ distributes to the outgoing roads at the junction $J$. In other words, if $X$ is the amount of traffic coming from road $I_i$ then $\alpha_{j,i}X$ is the desired amount of traffic going form $I_i$ towards road $I_j$.

As stated, for simplicity, we assume a \emph{simple network} with only one junction in the rest of this paper. All our definitions and theorems can then be extended straightforwardly to an arbitrary network $(\kI,\kJ)$.

\section{Discontinuous Galerkin method}
\label{sec:2}
We discretize the governing equation (\ref{LWR_problem}) using the \emph{discontinuous Galerkin} (DG) method. This method introduced by Reed and Hill in \cite{PrvniDG} represents a robust, reliable and accurate numerical method for the solution of first order hyperbolic problems. The DG method uses discontinuous piecewise polynomial approximations of the exact solution along with a suitable weak form of the governing equations and can thus be viewed as a combination of the the finite element and finite volume methods, cf. \cite{DG}, \cite{Stabilization}. One advantage of the DG method over standard finite elements is its robustness with respect to the Gibbs phenomenon. This occurs when a continuous approximation is used to approximate a discontinuous function -- these typically arise as solutions to nonlinear first order hyperbolic problems, such as those considered here.

In general, the DG method is described on a polygonal (polyhedral) domain $\Omega\subset\R^d$, $d\in\N$. Let $\kT_h$ be a partition of $\overline{\Omega}$ into a finite number of closed elements $K$ with mutually disjoint interiors, such that
\[
\overline{\Omega}=\bigcup_{K\in\kT_h}K.
\]
Since the traffic model is defined on a line, we consider $\Omega\subset\R$, $\Omega=(a,b)$. In the 1D case, an element $K$ is an interval $\left[ a_K,\ b_K\right]$, where $a_K$ and $b_K$ are boundary points of $K$. We set $h_K=\abs{b_K-a_K}$, $h=\max_{K\in\kT}h_K$. We denote the set of all boundary faces (points in 1D) of all elements by $\kF_h$. Further, we define the set of all inner points by
\[
\kF_h^I=\lbrace x\in\kF_h;\ x\in\Omega\rbrace 
\]
and the set of boundary points $\kF_h^B=\lbrace a,\ b\rbrace$. Obviously $\kF_h=\kF_h^I\cup\kF_h^B$. We use a suitable weak formulation of (\ref{LWR_problem}) on the \emph{broken Sobolev space} $H^k(\Omega,\ \kT_h)=\lbrace v;\ v|_K\in H^k(K),\ \forall K\in\kT_h\rbrace$, where $H^k(I)$, $k\in\N,$ is the Sobolev space over an interval $I$. Functions from this space will be approximated by discontinuous piecewise polynomial functions from the space
\[
S_h=\lbrace v;\ v|_K\in P^p(K),\ \forall K\in\kT_h\rbrace,
\]
where $P^p(K)$ denotes the space of all polynomials on $K$ of degree at most $p$.

For each point $x\in\kF_h^I$ there exist two neighbours $K_x^{-},\ K_x^{+}\in\kT_h$ such that $x=K_x^{-}\cap K_x^{+}$. Every function $v\in H^k(\Omega,\kT_h)$ is generally discontinuous at $x\in\kF_h^I$. Thus, for all $x\in\kF_h^I$, we introduce the following notation for traces and jumps:
\begin{equation*}
v^{-}(x)=\lim_{y\rightarrow x_-}v(y),\qquad v^{+}(x)=\lim_{y\rightarrow x_+}v(y),\qquad\left[ v\right]_x=v^{-}(x)-v^{+}(x).
\end{equation*}
In order to have consistent notation, in the point $x\in\kF_h^B$ we define
\begin{align*}
v^{+}(a)=\lim_{y\rightarrow a_+}v(y),\qquad&v(a)=-\left[ v\right]_a=v^{-}(a):=v^{+}(a),\\
v^{-}(b)=\lim_{y\rightarrow b_-}v(y),\qquad&v(b)=\left[ v\right]_b=v^{+}(b):=v^{-}(b).
\end{align*}
The definition of the jump $\left[ v\right]_a:=-v^{+}(a)$ or $\left[ v\right]_b:=v^{-}(b)$ may seem inconsistent with the definition on interior points. This notation is used due to the integration by parts in the following sections. For simplicity, if $\left[\cdot\right]_x$ appear in a sum of the form $\sum_{x\in\kF_h}\ldots$, we omit the index $x$ and write $\left[\cdot\right]$.

We formulate the DG method for first order hyperbolic problems of the form
\begin{align}
u_t+f(u)_x=0,&\qquad x\in\Omega,\ t\in(0,T),\label{DG_Hyperbolic_problem}\\
u=u_D,&\qquad x\in\kF_h^D,\ t\in(0,T),\\
u(x,0)=u_0(x),&\qquad x\in\Omega,\label{DG_Hyperbolic_problem_initial}
\end{align}
where the Dirichlet boundary condition $u_D:\kF_h^D\times(0,T)\rightarrow\R$ and the initial condition $u_0:\Omega\rightarrow\R$ are given functions. The Dirichlet boundary condition is prescribed only on the inlet $\kF_h^D\subseteq\kF_h^B$, respecting the direction of information propagation. The function $f\in\kC^1(\R)$ is called the \emph{convective flux}. Our aim is to seek a function $u:\Omega\times(0,T)\rightarrow\R$ such that (\ref{DG_Hyperbolic_problem})--(\ref{DG_Hyperbolic_problem_initial}) is satisfied. As we have seen, problem (\ref{DG_Hyperbolic_problem}) is the main part of macroscopic equations for traffic.

In order to derive the DG formulation of (\ref{DG_Hyperbolic_problem}), we multiply by a test function $\varphi\in H^1(\Omega,\kT_h)$ and integrate over an arbitrary element $K\in\kT_h$. Then we apply integration by parts and obtain
\begin{equation}
\int_K u_t\varphi\ \dd x-\int_K f(u)\varphi'\ \dd x+f(u(b_K,t))\varphi^{-}(b_K)-f(u(a_K,t))\varphi^{+}(a_K)=0.
\end{equation}
Finally, we sum over all $K\in\kT_h$ and obtain
\[
\int_\Omega u_t\varphi\ \dd x-\sum_{K\in\kT_h}\int_K f(u)\varphi'\ \dd x+\sum_{x\in\kF_h}f(u)\left[\varphi\right] =0.
\]

We wish to approximate $u$ by a function $u_h\in H^1(\Omega,\kT_h)$ which is in general discontinuous on $\kF_h$. Thus, we need to give proper meaning to the function $f(u_h)$ in points $x\in\kF_h$. We proceed similarly as in the finite volume method and use the approximation
\begin{equation}
f(u_h)\approx H(u_h^{-},u_h^{+}),
\end{equation}
where $H(\cdotp,\cdotp)$ is a \emph{numerical flux}, cf. \cite{DG}. Finally, we define the DG solution of problem (\ref{DG_Hyperbolic_problem}).

\begin{definition}[DG solution]
The function $u_h:\Omega\times(0,T)\rightarrow\R$ is called a DG finite element solution of hyperbolic problem (\ref{DG_Hyperbolic_problem})--(\ref{DG_Hyperbolic_problem_initial}) if the following properties hold:
\begin{itemize}
\item[(i)] $u_h\in\kC^1\left(\left[ 0,T\right];S_h\right)$.
\item[(ii)] $u_h(0)=u_{h0}$, where $u_{h0}$ denotes an $S_h$ approximation of the initial condition $u_0$.
\item[(iii)] $u_h=u_D$ for all $x\in\kF_h^D,\ t\in(0,T)$.
\item[(iv)] For all $\varphi\in S_h$ and for all $t\in\left( 0,T\right)$, $u_h$ satisfies
\begin{equation}\label{DG_Weak}
\int_\Omega (u_h)_t\varphi\ \dd x-\sum_{K\in\kT_h}\int_K f(u_h)\varphi'\ \dd x+\sum_{x\in\kF_h}H(u^{-}_h,u^{+}_h)\left[\varphi\right] =0.
\end{equation}
\end{itemize}
\end{definition}

\subsection{Godunov numerical flux}
\label{Subsec:Godunov}
In our implementation, we use the \emph{Godunov} numerical flux, cf. \cite{Stabilization}. This is defined as the flux $f$ evaluated at the exact solution of the Riemann problem with the piecewise defined initial condition $u^{-}$ and $u^{+}$. This can be shown to be equivalent to the more practical form, cf. \cite{Stabilization},
\begin{align}\label{Godunov_flux_orig}
H^{God}_{orig}\left(u^{-},u^{+}\right)=\begin{cases}
\min_{u^{-}\leq u\leq u^{+}} f(u),\qquad &\text{if }u^{-}<u^{+},\\
\max_{u^{+}\leq u\leq u^{-}} f(u),\qquad &\text{if }u^{-}\geq u^{+}.
\end{cases}
\end{align}

We call this form the \emph{original} form and for our purposes, we use an alternative form which is inspired by the maximum possible traffic flow approach (see Section \ref{Section_MPTF}) in the case of one incoming and one outgoing road.

\begin{definition}[Alternative form of the Godunov numerical flux]\label{def_Alternative_Godunov}
Let the convective flux $f$ have a global maximum at $u_*$ and $f$ is non--decreasing on the interval $(-\infty,u_*]$ and non--increasing on $[u_*,\infty)$. Then the Godunov numerical flux is defined as
\begin{equation}\label{Godunov_flux}
H^{God}\left(u^{-},u^{+}\right)=\min\left\lbrace f_{in}(u^{-}),f_{out}(u^{+})\right\rbrace,
\end{equation}
where
\begin{align*}
f_{in}(u^{-})=\begin{cases}
f(u^{-}),&\text{if }u^{-}<u_*,\\
f(u_*),&\text{if }u^{-}\geq u_*,
\end{cases}
\qquad f_{out}(u^{+})=\begin{cases}
f(u_*),&\text{if }u^{+}\leq u_*,\\
f(u^{+}),&\text{if }u^{+}>u_*.
\end{cases}
\end{align*}
\end{definition}

This can be interpreted as the maximal possible flow through the common boundary, where $f_{in}$ is the maximal possible inflow from the left element and $f_{out}$ is the maximal possible outflow to the right element. We can see $f_{in}$ and $f_{out}$ for Greenshields traffic flow in Figure \ref{Obr_f_in_f_out}.

\begin{figure}[t!]\centering
\includegraphics[height=2in]{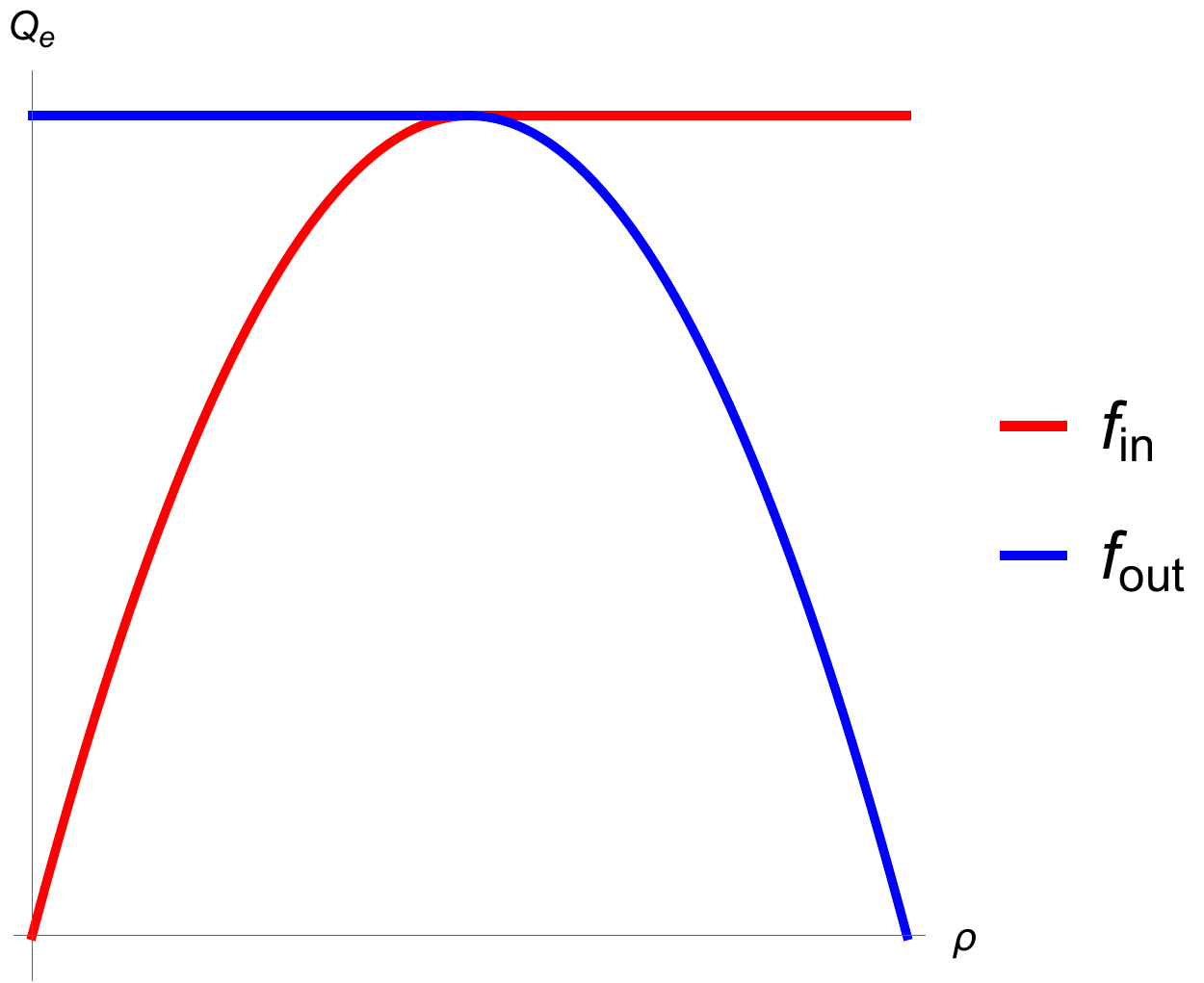}
\caption{\textcolor{red}{$f_{in}$} and \textcolor{blue}{$f_{out}$} for Greenshields traffic flow.}
\label{Obr_f_in_f_out}
\end{figure}

Formulas (\ref{Godunov_flux_orig}) and (\ref{Godunov_flux}) are equivalent, which is shown in the following lemma.
\begin{lemma}
If the convective flux $f$ attains its global maximum at $u_*$ and is non--decreasing on $(-\infty,u_*]$ and non--increasing on $[u_*,\infty)$, then $H^{God}_{orig}\left(u^{-},u^{+}\right)=H^{God}\left(u^{-},u^{+}\right)$ for all $u^{-},u^{+}\in\R$.
\end{lemma}

\begin{proof}
We divide the proof into six different cases. There are two main cases based on the ordering of $u^{-}$ and $u^{+}$. Then for each case we present three sub-cases, which place $u_*$ in different positions.
\begin{itemize}
\item[1. ] $u^{-}<u^{+}$
	\begin{itemize}
	\item[a) ] If $u^{-}<u^{+}\leq u_*$ then $H^{God}_{orig}=f(u^{-})$ and $H^{God}=\min\left\lbrace f(u^{-}),f(u_*)\right\rbrace=f(u^{-})$.
	\item[b) ] If $u^{-}\leq u_*<u^{+}$ then $H^{God}_{orig}=\min\left\lbrace f(u^{-}),f(u^{+})\right\rbrace$ and $H^{God}=\min\left\lbrace f(u^{-}),f(u^{+})\right\rbrace$.
	\item[c) ] If $u_*<u^{-}<u^{+}$ then $H^{God}_{orig}=f(u^{+})$ and $H^{God}=\min\left\lbrace f(u_*),f(u^{+})\right\rbrace=f(u^{+})$.
	\end{itemize}
\item[2. ] $u^{-}\geq u^{+}$
	\begin{itemize}
	\item[a) ] If $u^{+}\leq u^{-}< u_*$ then $H^{God}_{orig}=f(u^{-})$ and $H^{God}=\min\left\lbrace f(u^{-}),f(u_*)\right\rbrace=f(u^{-})$.
	\item[b) ] If $u^{+}\leq u_*\leq u^{-}$ then $H^{God}_{orig}=f(u_*)$ and $H^{God}=\min\left\lbrace f(u_*),f(u_*)\right\rbrace=f(u_*)$.
	\item[c) ] If $u_*<u^{+}\leq u^{-}$ then $H^{God}_{orig}=f(u^{+})$ and $H^{God}=\min\left\lbrace f(u_*),f(u^{+})\right\rbrace=f(u^{+})$.
	\end{itemize}
\end{itemize}
We showed that $H^{God}_{orig}\left(u^{-},u^{+}\right)=H^{God}\left(u^{-},u^{+}\right)$ in every possible situation.
\end{proof}

In the rest of the paper, the numerical flux $H(\cdotp,\cdotp)$ will always be the Godunov numerical flux written in the alternative form (\ref{Godunov_flux}).

\section{Numerical fluxes at junctions} \label{sec_numerical_fluxes_junctions}
In order to formulate the DG scheme on a simple network, we first need to construct the numerical fluxes at the junction. Such a numerical flux is considered in \cite{RKDG} and \cite{Networks} which is based on the assumption that the drivers wish to maximize the total flux through the junction while respecting the traffic distribution exactly. We discuss this approach in Section \ref{Section_MPTF}.

In this paper we take a different approach which has the advantage that it is simple and explicitly constructed for all junction types unlike that of \cite{RKDG} and \cite{Networks}, which leads to the solution of a linear programming problem. We will present two different constructions of the numerical fluxes at the junction based on the Godunov numerical flux in Sections \ref{sec_alpha_outside} and \ref{sec_alpha_inside}. We shall prove the basic properties of our construction and discuss the differences with the approach of \cite{RKDG} and \cite{Networks}.

\subsection{Maximum possible traffic flow}\label{Section_MPTF}
Based on the traffic distribution matrix, the authors of \cite{Networks} define the following admissible traffic solution at a junction, also used for numerical simulations in \cite{RKDG}.

\begin{definition}[Admissible traffic solution at a junction, following \cite{Networks}]\label{def_network_solution}
Let $ u=\T{( u_1,\ldots, u_{n+m})}$ be such that $ u_i(\cdot,t)$ is of bounded variation for every $t\geq 0$. Then $ u$ is called an \emph{admissible weak solution} of (\ref{LWR_network}) related to the matrix $A$ at the junction $J$ if the following properties hold:
\begin{itemize}
\item[(i)] $ u$ is a weak solution at the junction $J$.
\item[(ii)] $\f( u_j^{+}(a_j,\cdot))=\sum_{i=1}^n\alpha_{j,i}\f( u_i^{-}(b_i,\cdot)),$ for all $j=n+1,\ldots,n+m$.
\item[(iii)] $\sum_{i=1}^n \f( u_i^{-}(b_i,\cdot))$ is maximal subject to (i) and (ii).
\end{itemize}
\end{definition}

\begin{remark}
Condition (ii) simply states that traffic from incoming roads is distributed to outgoing roads according to the traffic distribution matrix. Condition (iii) is a mathematical formulation of the assumption made in \cite{Networks}, that respecting (ii), \uvoz{drivers choose so as to maximize the total flux} through the junction.
\end{remark}

One problem with the approach of \cite{Networks} and \cite{RKDG} is that explicitly constructing the fluxes from Definition \ref{def_network_solution} requires the solution of a linear programming problem on the incoming fluxes. This is done in \cite{Networks} for the purposes of constructing a Riemann solver at the junction and in \cite{RKDG} for the purposes of obtaining numerical fluxes at the junction in order to formulate the DG scheme. Closed-form solutions are provided in \cite{RKDG} in the special cases $n=1,m=2$ and  $n=2,m=1$ and $n=2,m=2$. In Section \ref{sec_numerical_fluxes_junctions}, we present an alternative construction of fluxes at the junction which has the advantage of a simple formulation for general $n,m$. We will give an interpretation of our construction, which shows that it is more suited for certain situations, giving more realistic behavior of the drivers, than the approach from Definition \ref{def_network_solution}. We compare the two approaches in Section \ref{sec_results}.

\subsubsection{One incoming and two outgoing roads}\label{Section_MPTF_example}
This example is important to us, because it inspires us in the construction of the $\alpha$-inside Godunov flux (see Section \ref{sec_alpha_inside}). We use the method described in \cite[Section 2.2]{RKDG} with our notation.

In this case, we have distribution coefficient $\alpha_{2,1}=\alpha$ and $\alpha_{3,1}=1-\alpha$. Then according to \cite{Networks}, we calculate the maximum possible inflow to the junction from the incoming road as
\begin{equation}\label{rov_max_possible_flow}
H_1(t)=\min\left\lbrace f_{in}( u_1^{-}(b_1,t)),\dfrac{f_{out}( u_2^{+}(a_2,t))}{\alpha},\dfrac{f_{out}( u_3^{+}(a_3,t))}{1-\alpha}\right\rbrace.
\end{equation}
The outflow from the junction to an outgoing road is calculated as $H_1$ multiplied by the corresponding distribution coefficient, i.e. $H_2(t)=\alpha H_1(t)$ and $H_3(t)=(1-\alpha)H_1(t)$.

\begin{remark}
\label{rem:blocked_junction}
We note, that traffic congestion on one of the outgoing roads influences the traffic flow to the second outgoing road. For example, when $f_{out}( u_2^{+})=0$, then $H_1=H_2=H_3=0$. Therefore, the intersection is completely blocked, even in the case when the other outgoing road $I_3$ is completely empty ($ u_3\equiv 0$). This is caused by the assumption in Definition \ref{def_network_solution} that drivers strictly adhere to their original preferences for the choice of the outgoing road, even on a congested intersection. Thus the flows to the outgoing roads $I_2, I_3$ must always be in the ratio $\alpha$ to $1-\alpha$, irrespective of the traffic situation. 
\end{remark}

\subsection{$\alpha$-outside Godunov flux}
\label{sec_alpha_outside}
In our previous paper \cite{Prvni_clanek}, we based the construction of the numerical flux at the junction on the Lax-Friedrichs numerical flux. In this paper we start from the Godunov numerical flux, which will have advantages that we will see in Section \ref{Section_Distribution_Error}. In either case, we construct the junction fluxes as follows.

At the junction, we consider an incoming road $I_i$ and an outgoing road $I_j$. If these roads were the only roads at the junction, i.e. if they were directly connected to each other, the (numerical) flux of traffic from $I_i$ to $I_j$ would simply be $H\big({u}_{hi}^{-}(b_i,t),{u}_{hj}^{+}(a_j,t)\big)$, where ${u}_{hi}$ and ${u}_{hj}$ are the DG solutions on $I_i$ and $I_j$, respectively. From the traffic distribution matrix, we know the ratios of the traffic flow distribution to the outgoing roads. Thus, we take the numerical flux $H_j(t)$ at the left point of the outgoing road $I_j$, i.e. at the junction, at time $t$ as
\begin{equation}\label{num_flux_out}
H_j(t):=\sum_{i=1}^{n} \alpha_{j,i} H\big({u}_{hi}^{-}(b_i,t),{u}_{hj}^{+}(a_j,t)\big),
\end{equation}
for $j=n+1,\ldots,n+m$. The numerical flux $H_j(t)$ can be viewed as the DG analogue of taking the combined traffic outflow $\sum_{i=1}^{n} \alpha_{j,i} \f\big({u}_i^{-}(b_i,t)\big)$ from all incoming roads and prescribing it as the inflow of traffic to the road $I_j$.

Similarly, we take the numerical flux $H_i(t)$ at the right point of the incoming road $I_i$, i.e. at the junction, at time $t$ as
\begin{equation}\label{num_flux_in}
H_i(t):=\sum_{j=n+1}^{n+m} \alpha_{j,i} H\big({u}_{hi}^{-}(b_i,t),{u}_{hj}^{+}(a_j,t)\big),
\end{equation}
for $i=1,\ldots,n$. Again, this can be viewed as an approximation of the traffic  flow $\sum_{j=n+1}^{n+m} \alpha_{j,i} \f\big({u}_j^{+}(a_j,t)\big)$ being prescribed as the outflow of traffic from $I_i$.

Now we can formulate the DG method for the simple network with one junction using the numerical fluxes defined in (\ref{num_flux_out}) and (\ref{num_flux_in}). Then the case of general networks is a straightforward generalization, where the aforementioned construction of numerical fluxes at junctions is applied on each junction separately.

We consider the DG formulation (\ref{DG_Weak}) on every incoming and outgoing road represented by the intervals $(a_i,b_i), i=1,\ldots,n$ and $(a_j,b_j), j=n+1,\ldots,n+m$, respectively. Since the DG method is applied on finite intervals, we replace the endpoints at $\pm\infty$ from Definition \ref{Def_network} by artificial inflow/outflow boundaries at finite points along with inflow Dirichlet data. For every interval $(a_k,b_k), k=1,\ldots,n+m$, we consider a partition $\kT_{hk}$ along with the corresponding discrete space $S_{hk}$. We write the DG formulation directly for the case of LWR models (\ref{LWR_problem}) with unknown density ${u}$ and flux $\f({u})$.

\begin{definition}[DG formulation on a simple network]\label{def_DG_solution_network}
We seek functions ${u}_{hk}\in\kC^1\left(\left[ 0,T\right];S_{hk}\right)$, $k=	1,\ldots,n+m$ satisfying the following.
\begin{itemize}
\item[(i)] \emph{Incoming roads:} For all $i=1\ldots,n$ and all $\varphi_i\in S_{hi}$
\begin{equation}\label{DG_Weak_inc}
\begin{split}
\int_{a_i}^{b_i} ({u}_{hi})_t\varphi_i\ \dd x-\sum_{K\in\kT_{hi}}\int_K \f({u}_{hi})\varphi_i'\ \dd x+ \sum_{x\in\kF_{hi}^I}H\big({u}_{hi}^{-},{u}_{hi}^{+}\big)\left[\varphi_i\right]&\\ +H_i\varphi_i^{-}(b_i) -H\big({u}_{Di},{u}_{hi}^{+}(a_i)\big)\varphi_i^{+}(a_i)&=0,
\end{split}
\end{equation}
where $H_i=H_i(t)$ is the numerical flux defined in (\ref{num_flux_in}) and ${u}_{Di}$ is the Dirichlet datum corresponding to the left artificial inflow boundary point $a_i$ of $(a_i,b_i)$.
\item[(ii)] \emph{Outgoing roads:} For all $j=n+1,\ldots,n+m$ and all $\varphi_j\in S_{hj}$
\begin{equation}\label{DG_Weak_out}
\begin{split}
\int_{a_j}^{b_j} ({u}_{hj})_t\varphi_j\ \dd x-\sum_{K\in\kT_{hj}}\int_K \f({u}_{hj})\varphi_j'\ \dd x+ \sum_{x\in\kF_{hj}^I}H\big({u}_{hj}^{-},{u}_{hj}^{+}\big)\left[\varphi_j\right]&\\  \quad +H\big({u}_{hj}^{-}(b_j),{u}_{Dj}\big)\varphi_j^{-}(b_j) -H_j\varphi_j^{+}(a_j)&=0,
\end{split}
\end{equation}
where $H_j=H_j(t)$ is the numerical flux defined in (\ref{num_flux_out}).
\end{itemize}
\end{definition}

\begin{remark}
We note that the choice of the arguments in the numerical flux at the artificial boundary point $b_j$ in (\ref{DG_Weak_out}) corresponds to an outflow boundary condition. Typically, ${u}_{Dj}={u}_{hj}^{-}(b_j)$. This term could be rewritten using the original physical flux $\f({u}_{hj}^{-}(b_j))$ due to consistency of the numerical flux $H$.
\end{remark}

Definitions (\ref{num_flux_out}) and (\ref{num_flux_in}) are independent of the specific choice of the numerical flux $H$. In the paper \cite{Prvni_clanek}, we used the Lax-Friedrichs numerical flux, while in this paper we use Godunov's flux. Since the distribution coefficients $\alpha_{j,i}$ are outside of the numerical flux $H$ in (\ref{num_flux_out}) and (\ref{num_flux_in}), we call this construction the \emph{$\alpha$-outside Godunov flux}, as opposed to the following section.

\subsection{$\alpha$-inside Godunov flux}\label{sec_alpha_inside}
When comparing the maximum possible traffic flow (\ref{rov_max_possible_flow}) with the $\alpha$-outside Godunov flux  (\ref{num_flux_in}), we find the main difference in the position of the distribution coefficient. While in (\ref{num_flux_in}) the traffic distribution coefficient is outside of the minimization (\ref{Godunov_flux}) defining the Godunov flux, in (\ref{rov_max_possible_flow}) this coefficient is inside the minimization defining the Godunov-like flux. This leads to the idea of moving the distribution coefficients in (\ref{num_flux_out}) and (\ref{num_flux_in}) inside the Godunov numerical fluxes. To this end we define an auxiliary Godunov-like flux with a third variable for the traffic distribution coefficient.

\begin{definition}[Godunov numerical flux with three variables]\label{def_three_variables_Godunov}
The \emph{Godunov numerical flux with three variables} is defined as
\begin{equation}\label{Godunov_flux_3}
H^{God}\left(u^{-},u^{+},\alpha\right)=\min\left\lbrace\alpha f_{in}(u^{-}),f_{out}(u^{+})\right\rbrace,
\end{equation}
where $f_{in}(u^{-})$ and $f_{out}(u^{+})$ are defined as in Definition \ref{def_Alternative_Godunov}.
\end{definition}

The reason why we put the distribution coefficient in front of $f_{in}$, is the representation of the real flow from the incoming road. Only $\alpha_{j,i}f_{in}({u}_i^{-}(b_i,t))$ cars per time unit want to go from incoming road $i$ to outgoing road $j$. So it makes sense to use $\alpha_{j,i}f_{in}$ in the definition of the Godunov flux instead of $f_{in}$ as in (\ref{Godunov_flux}). For simplicity, we drop the superscript ``$God$" and define the numerical flux with three variables $H(\cdotp,\cdotp,\cdotp)$ as the flux (\ref{Godunov_flux_3}) in the rest of this paper.

Now we can take the numerical flux $H_j(t)$ with $\alpha$-inside at the left point of the outgoing road $I_j$, i.e. at the junction, at time $t$ as
\begin{equation}\label{num_flux_inside_out}
H_j(t):=\sum_{i=1}^{n} H\big({u}_{hi}^{-}(b_i,t),{u}_{hj}^{+}(a_j,t),\alpha_{j,i}\big),
\end{equation}
for $j=n+1,\ldots,n+m$. Similarly, we take the numerical flux $H_i(t)$ with $\alpha$-inside at the right point of the incoming road $I_i$, i.e. at the junction, as
\begin{equation}\label{num_flux_inside_in}
H_i(t):=\sum_{j=n+1}^{n+m} H\big({u}_{hi}^{-}(b_i,t),{u}_{hj}^{+}(a_j,t),\alpha_{j,i}\big),
\end{equation}
for $i=1,\ldots,n$. We can use the same DG formulation as in Definition \ref{def_DG_solution_network} with $H_i$ defined as (\ref{num_flux_inside_in}) and $H_j$ defined as (\ref{num_flux_inside_out}).

\subsection{One incoming and two outgoing roads}
We use the same example as in Section \ref{Section_MPTF_example} with one incoming and two outgoing road and compare all three approaches. Our aim is to identify and discuss the differences between them and describe their behavior. For simplicity, we use the notation $f_{in}^{(1)}:=f_{in}({u}_1^{-}(b_1,t))$, $f_{out}^{(2)}:=f_{out}({u}_2^{+}(a_2,t))$ and $f_{out}^{(3)}:=f_{out}({u}_3^{+}(a_3,t))$.

In the case of the $\alpha$-outside Godunov flux, we calculate the inflow to the junction from the incoming road as
\begin{equation}\label{rov_example_alpha_out}
H_1(t)=\alpha \min\left\lbrace f_{in}^{(1)},f_{out}^{(2)}\right\rbrace + \left(1-\alpha\right) \min\left\lbrace f_{in}^{(1)},f_{out}^{(3)}\right\rbrace.
\end{equation}
We compare this flux value to that obtained by the maximum possible flow (\ref{rov_max_possible_flow}). At first glance equations (\ref{rov_max_possible_flow}) and (\ref{rov_example_alpha_out}) seem completely different, cf. Table \ref{tab_comparison}. On the other hand, if $f_{in}$ is less than both $f_{out}^{(2)}$ and $f_{out}^{(3)}$ then equations (\ref{rov_max_possible_flow}) and (\ref{rov_example_alpha_out}) give the same value $f_{in}$. The outflows from the junction are the individual terms in the right hand side of (\ref{rov_example_alpha_out}). Hence, $H_2(t)=\alpha \min\lbrace f_{in}^{(1)},f_{out}^{(2)}\rbrace$ and $H_3(t)=(1-\alpha) \min\lbrace f_{in}^{(1)},f_{out}^{(3)}\rbrace$. If we compare $H_2$ and $H_3$ from the maximum possible flow and $\alpha$-outside Godunov flux, we can see the distribution coefficient $\alpha$ and $1-\alpha$ in front of the minimization in both cases. Again, if $f_{in}$ is less than both $f_{out}^{(2)}$ and $f_{out}^{(3)}$, both of the approaches give the same values of $H_2$ and $H_3$. But if the minimizer is $f_{out}^{(2)}$ and $f_{out}^{(2)}$ divided by the corresponding distribution coefficient, respectively, then the maximum possible flow gives us $H_2=f_{out}^{(2)}$ while $\alpha$-outside Godunov gives us $H_2=\alpha f_{out}^{(2)}$, which is a lower flux value than the maximum possible traffic flow. We can sum up these considerations informally in the statement that if the outgoing roads are emptier than the incoming road then the maximum possible traffic flow and the  $\alpha$-outside Godunov flux coincide. Once one of the outgoing roads is fuller than the incoming one, the two approaches differ, the latter one giving a smaller traffic flux. 

In the case of the $\alpha$-inside Godunov flux, we calculate the inflow to the junction from the incoming road as
\begin{equation}\label{rov_example_alpha_in}
H_1(t)=\min\left\lbrace\alpha f_{in}^{(1)},f_{out}^{(2)}\right\rbrace + \min\left\lbrace (\left(1-\alpha\right)f_{in}^{(1)},f_{out}^{(3)}\right\rbrace.
\end{equation}
This approach is somewhere between the $\alpha$-outside Godunov flux and maximum possible flow, see Table \ref{tab_comparison}. Again, if $f_{in}$ is less than both $f_{out}^{(2)}$ and $f_{out}^{(3)}$, equations (\ref{rov_max_possible_flow}), (\ref{rov_example_alpha_out}) and (\ref{rov_example_alpha_in}) give us same value. But if one of the $f_{out}$ is the minimizer, we typically get three different values. Also in this case the outflows from the junction are the individual terms in (\ref{rov_example_alpha_in}), hence $H_2(t)=\min\lbrace\alpha f_{in}^{(1)},f_{out}^{(2)}\rbrace$ and $H_3(t)=\min\lbrace (1-\alpha)f_{in}^{(1)},f_{out}^{(3)}\rbrace$. Here is the main difference between the $\alpha$-outside and $\alpha$-inside Godunov fluxes. Of course, when $f_{in}$ is less than both $f_{out}^{(2)}$ and $f_{out}^{(3)}$, we get the same values. But if $f_{out}^{(2)}$ is the minimizer, $\alpha$-inside Godunov gives us the same $H_2$ as the maximum possible flow ($H_1$ and $H_3$ are typically different). This is why we say that $\alpha$-inside Godunov lies between the other two approaches and takes positives from both.

\begin{table}[t!]\centering
\begin{tabular}{c||c|c|c}
\toprule
Approach & $H_1$ & $H_2$ & $H_3$ \\
\hline
\midrule
Maximum & \multirow{2}{*}{$\min\lbrace f_{in}^{(1)},\frac{f_{out}^{(2)}}{\alpha},\frac{f_{out}^{(3)}}{1-\alpha}\rbrace$} & \multirow{2}{*}{$\alpha\min\lbrace f_{in}^{(1)},\frac{f_{out}^{(2)}}{\alpha},\frac{f_{out}^{(3)}}{1-\alpha}\rbrace$} & \multirow{2}{*}{$(1-\alpha)\min\lbrace f_{in}^{(1)},\frac{f_{out}^{(2)}}{\alpha},\frac{f_{out}^{(3)}}{1-\alpha}\rbrace$} \\
possible & & & \\
\hline
$\alpha$-outside & \multirow{2}{*}{\begin{footnotesize}$\alpha\min\lbrace f_{in}^{(1)},f_{out}^{(2)}\rbrace+(1-\alpha)\min\lbrace f_{in}^{(1)},f_{out}^{(3)}\rbrace$\end{footnotesize}} & \multirow{2}{*}{$\alpha\min\lbrace f_{in}^{(1)},f_{out}^{(2)}\rbrace$} & \multirow{2}{*}{$(1-\alpha)\min\lbrace f_{in}^{(1)},f_{out}^{(3)}\rbrace$} \\
Godunov & & & \\
\hline
$\alpha$-inside & \multirow{2}{*}{\begin{footnotesize}$\min\lbrace\alpha f_{in}^{(1)},f_{out}^{(2)}\rbrace+\min\lbrace (1-\alpha)f_{in}^{(1)},f_{out}^{(3)}\rbrace$\end{footnotesize}} & \multirow{2}{*}{$\min\lbrace\alpha f_{in}^{(1)},f_{out}^{(2)}\rbrace$} & \multirow{2}{*}{$\min\lbrace (1-\alpha)f_{in}^{(1)},f_{out}^{(3)}\rbrace$} \\
Godunov & & & \\
\hline
\end{tabular}
\caption{Comparison of the three approaches on the example with one incoming and two outgoing roads.}\label{tab_comparison}
\end{table}

\section{Properties}
\label{sec:5}
In this section, we look at the basic properties of the numerical fluxes at junctions that we considered in Section \ref{sec_numerical_fluxes_junctions}. Namely, the Rankine-Hugoniot conditions and the satisfaction of the traffic distribution according to the coefficients in the traffic distribution matrix.

\subsection{Rankine-Hugoniot condition}
First, we show that our Godunov-based fluxes satisfy the discrete analogues of the Rankine-Hugoniot condition (\ref{Junction_equation}), which leads to conservativity of the resulting DG scheme.

\begin{lemma}[Discrete Rankine--Hugoniot condition for $\alpha$-outside Godunov flux]\label{Lemma_RankineHugoniot_a-out}
The numerical fluxes (\ref{num_flux_out}) and (\ref{num_flux_in}) with $\alpha$ outside satisfy the discrete version of the Rankine--Hugoniot condition (\ref{Junction_equation}):
\[
\sum_{i=1}^n H_i(t)=\sum_{j=n+1}^{n+m} H_j(t).
\]
\end{lemma}
\begin{proof}
From the definition of $H_i$ and $H_j$, we immediately obtain
\begin{align*}
\sum_{i=1}^n H_i(t)&=\sum_{i=1}^n\sum_{j=n+1}^{n+m} \alpha_{j,i} H\big({u}_{hi}^{-}(b_i,t),{u}_{hj}^{+}(a_j,t)\big)\\
&=\sum_{j=n+1}^{n+m}\sum_{i=1}^n \alpha_{j,i} H\big({u}_{hi}^{-}(b_i,t),{u}_{hj}^{+}(a_j,t)\big)=\sum_{j=n+1}^{n+m} H_j(t).
\end{align*}
\end{proof}

\begin{lemma}[Discrete Rankine--Hugoniot condition for $\alpha$-inside Godunov flux]\label{Lemma_RankineHugoniot_a-in}
The numerical fluxes (\ref{num_flux_inside_out}) and (\ref{num_flux_inside_in}) with $\alpha$ inside satisfy the discrete version of the Rankine--Hugoniot condition (\ref{Junction_equation}):
\[
\sum_{i=1}^n H_i(t)=\sum_{j=n+1}^{n+m} H_j(t).
\]
\end{lemma}
\begin{proof}
From the definition of $H_i$ and $H_j$, we immediately obtain
\begin{align*}
\sum_{i=1}^n H_i(t)&=\sum_{i=1}^n\sum_{j=n+1}^{n+m} H\big({u}_{hi}^{-}(b_i,t),{u}_{hj}^{+}(a_j,t),\alpha_{j,i}\big)\\
&=\sum_{j=n+1}^{n+m}\sum_{i=1}^n H\big({u}_{hi}^{-}(b_i,t),{u}_{hj}^{+}(a_j,t),\alpha_{j,i}\big)=\sum_{j=n+1}^{n+m} H_j(t).
\end{align*}
\end{proof}

The previous two lemmas allow us to prove that for the DG scheme, the total number of vehicles in the network is conserved (modulo inlet and outlet boundary conditions), which is the basic property we expect from the conservation laws. Since the DG scheme is naturally conservative on each individual road, the question boils down to conservativity of the scheme at the junction.

\begin{corollary}[Conservation property of the DG scheme]\label{Corollary_Conservation_property}
Let ${u}_{Dj}={u}_{hj}^{-}(b_j)$. The DG scheme from Definition \ref{def_DG_solution_network} conserves the total number of vehicles in the network in the sense that
\begin{equation}
\frac{\dd}{\dd t}\sum_{k=1}^{n+m}\int_{a_k}^{b_k}{u}_{hk}\,\dd x= \sum_{i=1}^n H\big({u}_{Di},{u}_{hi}^{+}(a_i)\big) -\sum_{j=n+1}^{n+m} H\big({u}_{hj}^{-}(b_j),{u}_{Dj}\big)
\nonumber
\end{equation}
for both $\alpha$--outside and $\alpha$--inside Godunov fluxes.
\end{corollary}
\begin{proof}
We set all test functions as $\varphi_k\equiv 1$ for all $k=1,\ldots,n+m$ and sum together all of the equations (\ref{DG_Weak_inc}) and (\ref{DG_Weak_out}) for all $i$ and $j$. We get
\begin{equation}
\begin{split}
\frac{\dd}{\dd t}\sum_{k=1}^{n+m}&\int_{a_k}^{b_k}{u}_{hk}\,\dd x +\sum_{i=1}^n H_i -\sum_{j=n+1}^{n+m} H_j +\sum_{j=n+1}^{n+m} H\big({u}_{hj}^{-}(b_j),{u}_{Dj}\big) -\sum_{i=1}^n H\big({u}_{Di},{u}_{hi}^{+}(a_i)\big) =0.
\end{split}
\nonumber
\end{equation}
The second and third terms cancel one another since $\sum_i H_i-\sum_j H_j=0$ due to Lemma \ref{Lemma_RankineHugoniot_a-out} or \ref{Lemma_RankineHugoniot_a-in}. This completes the proof.
\end{proof}

\subsection{Traffic distribution error} \label{Section_Distribution_Error}
The purpose of the traffic distribution coefficients $\alpha_{j,i}$ from Definition \ref{def:Traffic_distribution_matrix} is to describe the preferences of the drivers from each incoming road for which outgoing road they want to take. In other words, how the traffic from each incoming road is distributed among the outgoing roads. In the end this means we want to satisfy condition (ii) from Definition \ref{def_network_solution} in terms of the numerical fluxes at the junction:
\begin{equation}
\label{traffic_distribution_exact}
H_j(t)=\sum_{i=1}^n \alpha_{j,i} H_i(t),\quad j=n+1,\ldots,n+m
\end{equation}
The maximum possible flux approach from \cite{Networks} is based on maximizing the total flux through the junction while satisfying (\ref{traffic_distribution_exact}) exactly under any circumstances. In this section we show that in the case of our Godunov-based fluxes, this relation is satisfied if, loosely speaking, the junction is not congested. In other words, whenever the outgoing roads can accept the incoming traffic, the drivers drive according to their original preferences. If any of the outgoing roads cannot accept the incoming traffic, relation (\ref{traffic_distribution_exact}) is no longer satisfied exactly, but with a small error (\emph{traffic distribution error}). We interpret this as a natural behavior of human drivers -- when one of the preferred outgoing roads is full, some drivers will change their original preferences and take a different route. In this context, strictly adhering to (\ref{traffic_distribution_exact}) irrespective of the traffic situation would correspond to non-human drivers, e.g. autonomous vehicles with a prescribed course which cannot be changed.

Another interpretation is the presence of dedicated turning lanes in front of the junction. These allow other cars to pass the standing vehicles which want to go to a congested outgoing road. In a single-lane road, even one standing vehicle can block the whole road if it cannot proceed to its desired outgoing road. Such a vehicle blocks the way for other drivers, even those who could otherwise proceed since their desired outgoing roads are free. This was discussed in Remark \ref{rem:blocked_junction} for the maximum possible flux, where a single congested road blocks the entire junction.

In this section, we analyze when the traffic distribution condition (\ref{traffic_distribution_exact}) is satisfied exactly for our Godunov-based fluxes. In the following lemmas, we express the traffic distribution error for these fluxes.

\begin{lemma}[Traffic distribution error for $\alpha$--outside Godunov]\label{theorem_distribution_out}
The numerical fluxes (\ref{num_flux_out}) and (\ref{num_flux_in}) with $\alpha$ outside satisfy
\[
H_j(t)=\sum_{i=1}^n \alpha_{j,i} H_i(t) +E_j(t)
\]
for all $j=n+1,\ldots,n+m$, where the error term is
\begin{equation}\label{rov_traffic_distribution_error_out}
E_j(t)=\sum_{i=1}^n\sum_{\substack{l=n+1\\l\neq j}}^{n+m} \alpha_{j,i}\alpha_{l,i} \big(H_{i,j}(t) -H_{i,l}(t)\big),
\end{equation}
where $H_{i,j}(t):=H\big({u}_{hi}^{-}(b_i,t),{u}_{hj}^{+}(a_j,t)\big)$.
\end{lemma}
\begin{proof}
By definition (\ref{num_flux_out}),
\[
H_j(t)=\sum_{i=1}^n\alpha_{j,i}H_{i,j}(t) =\sum_{i=1}^n\alpha_{j,i}H_{i}(t) +\underbrace{\sum_{i=1}^n\alpha_{j,i}\big(H_{i,j}(t)-H_i(t)\big)}_{E_j(t)},
\]
where $E_j(t)$ is the error term which we will show has the form (\ref{rov_traffic_distribution_error_out}): By Definition (\ref{num_flux_in}), we have
\begin{equation*}
\begin{split}
E_j(t)&= \sum_{i=1}^n\alpha_{j,i}\Big(H_{i,j}(t)-\sum_{l=n+1}^{n+m}\alpha_{l,i}H_{i,l}(t)\Big)= \sum_{i=1}^n\alpha_{j,i}\sum_{l=n+1}^{n+m}\alpha_{l,i}\big(H_{i,j}(t)-H_{i,l}(t)\big)
\\ &=\sum_{i=1}^n\sum_{\substack{l=n+1\\l\neq j}}^{n+m} \alpha_{j,i}\alpha_{l,i} \big(H_{i,j}(t) -H_{i,l}(t)\big),
\end{split}
\end{equation*}
since $\sum_{l=n+1}^{n+m}\alpha_{l,i}=1$ due to (\ref{sum_alpha}). This completes the proof.
\qed\end{proof}

\begin{lemma}[Traffic distribution error for $\alpha$--inside Godunov]\label{theorem_distribution_in}
The numerical fluxes (\ref{num_flux_inside_out}) and (\ref{num_flux_inside_in}) with $\alpha$ inside satisfy
\[
H_j(t)=\sum_{i=1}^n \alpha_{j,i} H_i(t) +E_j(t)
\]
for all $j=n+1,\ldots,n+m$, where the error term is
\begin{equation}\label{rov_traffic_distribution_error_in}
E_j(t)=\sum_{i=1}^n\sum_{\substack{l=n+1\\l\neq j}}^{n+m} \left(\alpha_{l,i} H_{i,j}(t) -\alpha_{j,i} H_{i,l}(t)\right),
\end{equation}
where $H_{i,j}(t):=H\big({u}_{hi}^{-}(b_i,t),{u}_{hj}^{+}(a_j,t),\alpha_{j,i}\big)$.
\end{lemma}
\begin{proof}
By definition (\ref{num_flux_inside_out}),
\[
H_j(t)=\sum_{i=1}^n H_{i,j}(t) =\sum_{i=1}^n\alpha_{j,i}H_{i}(t) +\underbrace{\sum_{i=1}^n\big(H_{i,j}(t)-\alpha_{j,i} H_i(t)\big)}_{E_j(t)},
\]
where $E_j(t)$ is the error term which we will show has the form (\ref{rov_traffic_distribution_error_in}): By Definition (\ref{num_flux_inside_in}), we have
\begin{equation*}
E_j(t)= \sum_{i=1}^n\Big(H_{i,j}(t)-\alpha_{j,i}\sum_{l=n+1}^{n+m}H_{i,l}(t)\Big)=\sum_{i=1}^n\sum_{\substack{l=n+1\\l\neq j}}^{n+m} \big(\alpha_{l,i} H_{i,j}(t) -\alpha_{j,i} H_{i,l}(t)\big),
\end{equation*}
since $\sum_{l=n+1}^{n+m}\alpha_{l,i}=1$ due to (\ref{sum_alpha}). This completes the proof.
\end{proof}

Now we discuss general situations when the traffic distribution errors are zero. We note that we have also analyzed the traffic distribution errors in our previous paper \cite{Prvni_clanek}, where the numerical fluxes at the junction were based on the Lax-Friedrichs flux instead of Godunov. In that case, the errors were in general always nonzero but small. We therefore view the Godunov-like approach of this paper as more natural, since the traffic distribution relation (\ref{traffic_distribution_exact}) is satisfied exactly in many situation. Namely, in Theorem \ref{theorem_distribution2_in} we show that if the density on all the outgoing roads is smaller than $u_*$, i.e. the density where the traffic flow is maximal, then the traffic is distributed according to (\ref{traffic_distribution_exact}).

In the following, we shall consider a fixed time $t$ and omit the argument $t$ from the functions ${u}_h, H_{i,j}, E_j$ and similar, in order to simplify the notation.

\begin{theorem}[Zero traffic distribution error] \label{theorem_distribution2_in}
Let ${u}_{hj}^{+}(a_j)\leq u_*$ for all $j\in\lbrace n+1,\ldots,n+m\rbrace$. Then $E_j=0$ for all $j\in\lbrace n+1,\ldots,n+m\rbrace$ for both the $\alpha$--inside and $\alpha$--outside Godunov fluxes. 
\end{theorem}
\begin{proof}
For the $\alpha$-outside flux, due to the form (\ref{rov_traffic_distribution_error_out}) of the error term, if $H_{i,j}=H_{i,l}$ for all $i,j,l$, then $E_j=0$. This happens, in general, when the inflow term is the minimizer in (\ref{Godunov_flux}) in both numerical fluxes, i.e. $H_{i,j}=H_{i,l}=f_{in}({u}_{hi}^{-}(b_i))$. The simplest general case when this happens is if $f_{out}({u}_{hj}^{+}(a_j)) =f_{out}({u}_{hl}^{+}(a_l))=f(u_*)\ge f_{in}({u}_{hi}^{-}(b_i))$, or in other words if ${u}_{hj}^{+}(a_j)\leq u_*$ and ${u}_{hl}^{+}(a_l)\leq u_*$ for all $l,j\in\lbrace n+1,\ldots,n+m\rbrace$.

For the $\alpha$-inside flux, due to the form (\ref{rov_traffic_distribution_error_out}) of the error term, if $\alpha_{l,i}H_{i,j}=\alpha_{j,i}H_{i,l}$ for all $i,j,l$, then $E_j=0$. Again,if the inflow term is the minimizer in (\ref{Godunov_flux_3}) in both numerical fluxes, we get $\alpha_{l,i}H_{i,j}=\alpha_{l,i}\alpha_{j,i}f_{in}({u}_{hi}^{-}(b_i)) =\alpha_{j,i}H_{i,l}$. This occurs if $f_{out}({u}_{hj}^{+}(a_j)) =f(u_*)\ge \alpha_{j,i}f_{in}({u}_{hi}^{-}(b_i))$, or in other words if ${u}_{hj}^{+}(a_j)\leq u_*$ for all $j\in\lbrace n+1,\ldots,n+m\rbrace$.
\end{proof}

Theorem \ref{theorem_distribution2_in} states that the traffic is distributed exactly according to the original preferences if the outgoing roads are sufficiently free. The result is valid for both of the presented variants of the Godunov flux. If we take into account the specific form of the flux, where $f(0)=f(u_{\max}=0$, we can perform a finer analysis, which allows heavier traffic through the junction, while still satisfying (\ref{traffic_distribution_exact}) exactly.

\begin{theorem}[Zero traffic distribution error for $\alpha$--outside]\label{theorem_distribution2a_out}
Assume that for each $i\in\lbrace 1,\ldots,n\rbrace$ one of the following conditions is satisfied:
\begin{enumerate}
\item ${u}_{hi}^{-}(b_i)\ge u^*$ and ${u}_{hj}^{+}(a_j)\leq u_*$ for all $j\in\lbrace n+1,\ldots,n+m\rbrace$.
\item  ${u}_{hi}^{-}(b_i)< u^*$ and ${u}_{hj}^{+}(a_j)\leq \tilde{u}$ for all $j\in\lbrace n+1,\ldots,n+m\rbrace$, where $\tilde{u}>u^*$ is such that $f(\tilde{u})=f({u}_{hi}^{-}(b_i))$.
\end{enumerate}
Then $E_j=0$ for all $j\in\lbrace n+1,\ldots,n+m\rbrace$. 
\end{theorem}
\begin{proof}
As in the proof of Theorem \ref{theorem_distribution2_in}, we wish to have $H_{i,j}=H_{i,l}$ for all $i,j,l$. And again we achieve this by assuming the inflow term is the minimizer in both numerical fluxes in (\ref{Godunov_flux}), i.e. $H_{i,j}=H_{i,l}=f_{in}({u}_{hi}^{-}(b_i))$. Consult Figure  \ref{Vacek_Figure_Distribution_error}(a) in the following.

\textbf{Case 1.} If ${u}_{hi}^{-}(b_i)\ge u^*$ then $f_{in}({u}_{hi}^{-}(b_i)) = f(u^*)$ and for $H_{i,j}=H_{i,l}=f_{in}(u^*)$ to hold, we must necessarily have also $f_{out}({u}_{hj}^{+}(a_j)) = f(u^*)$, i.e. ${u}_{hj}^{+}(a_j)\leq u_*$ for all $j\in\lbrace n+1,\ldots,n+m\rbrace$.

\textbf{Case 2.} Since ${u}_{hi}^{-}(b_i)< u^*$ and since $f$ is non-increasing on $[u^*,\infty)$, there exists $\tilde{u}>u^*$ such that $f(\tilde{u})=f({u}_{hi}^{-}(b_i))$, cf. Figure \ref{Vacek_Figure_Distribution_error}(a). Then if ${u}_{hj}^{+}(a_j)\leq\tilde{u}$ for all $j\in\lbrace n+1,\ldots,n+m\rbrace$, we have $f_{out}({u}_{hj}^{+}(a_j))\ge f(\tilde{u})= f({u}_{hi}^{-}(b_i)) =f_{in}({u}_{hi}^{-}(b_i))$. Therefore, $H_{i,j}=f_{in}({u}_{hi}^{-}(b_i))$ for all $j\in\lbrace n+1,\ldots,n+m\rbrace$, hence $H_{i,j}=H_{i,l}$ for all $j,l$.
\end{proof}

\begin{figure}[b!]\centering
\subfloat[$\alpha$--outside.]{\label{Vacek_Figure_Distribution_error_a}\includegraphics[height=2.5in]{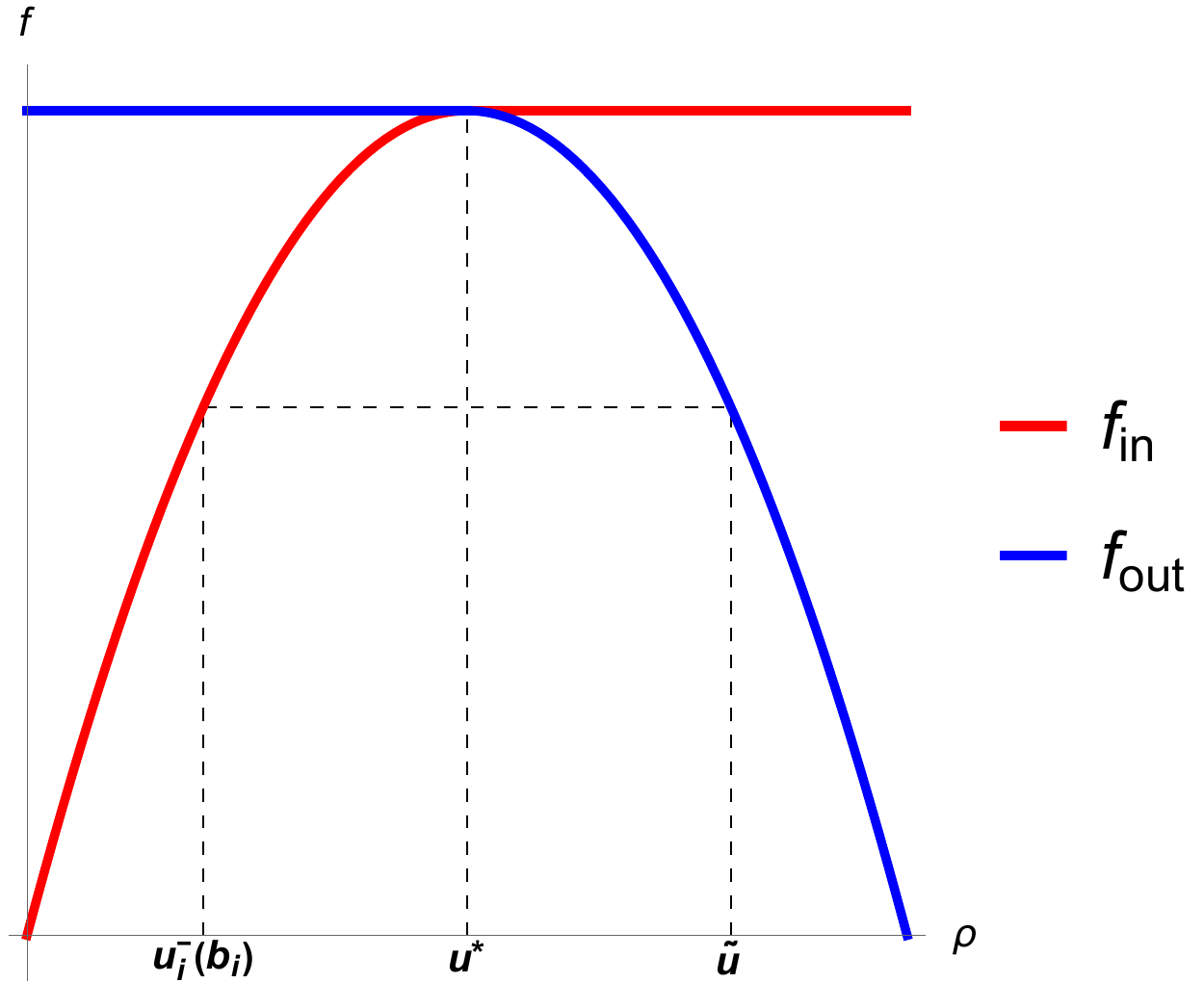}}
\hspace{0pt}
\subfloat[$\alpha$--inside.]{\label{Vacek_Figure_Distribution_error_b}\includegraphics[height=2.5in]{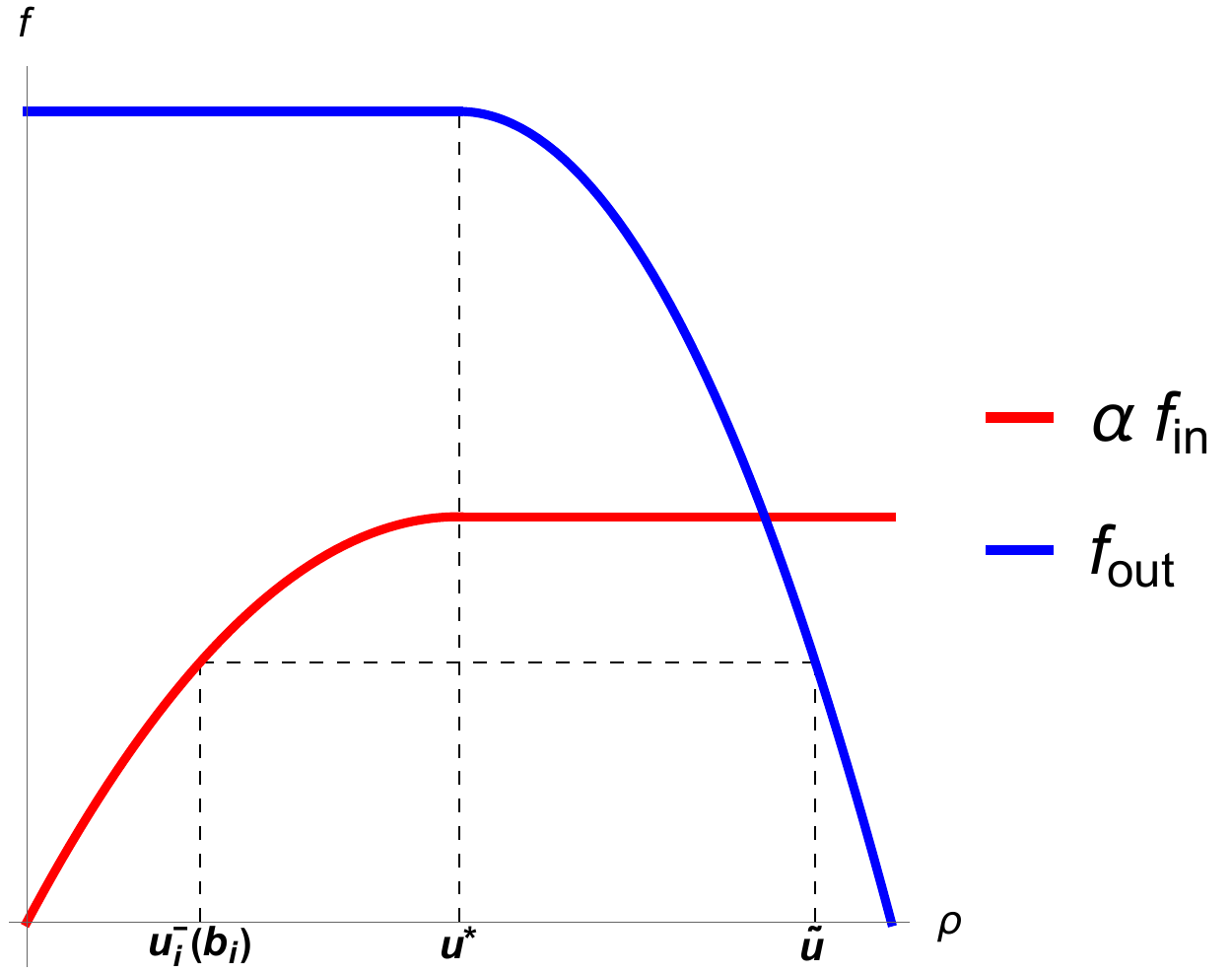}}
\caption{Illustration of Theorems \ref{theorem_distribution2a_out} and \ref{theorem_distribution2a_in}.}
\label{Vacek_Figure_Distribution_error}
\end{figure}

\begin{theorem}[Zero traffic distribution error for $\alpha$--inside]\label{theorem_distribution2a_in}
Assume that for each $i\in\lbrace 1,\ldots,n\rbrace$ one of the following conditions is satisfied:
\begin{enumerate}
\item ${u}_{hi}^{-}(b_i)\ge u^*$ and ${u}_{hj}^{+}(a_j)\leq \tilde{u}_{j}$ for all $j\in\lbrace n+1,\ldots,n+m\rbrace$, where $\tilde{u}_{j}\ge u^*$ is such that $f(\tilde{u}_{j})=\alpha_{j,i}f(u^*)$.
\item  ${u}_{hi}^{-}(b_i)< u^*$ and ${u}_{hj}^{+}(a_j)\leq \tilde{u}_{j}$ for all $j\in\lbrace n+1,\ldots,n+m\rbrace$, where $\tilde{u}_{j}>u^*$ is such that $f(\tilde{u}_{j})=\alpha_{j,i}f({u}_{hi}^{-}(b_i))$.
\end{enumerate}
Then $E_j=0$ for all $j\in\lbrace n+1,\ldots,n+m\rbrace$. 
\end{theorem}
\begin{proof}
Due to Lemma \ref{theorem_distribution_in}, we wish to have $\alpha_{l,i}H_{i,j}=\alpha_{j,i}H_{i,l}$ for all $i,j,l$. And again we achieve this by assuming the inflow term is the minimizer in both numerical fluxes in the Godunov flux (\ref{Godunov_flux_3}), since then we will have $\alpha_{l,i}H_{i,j}=\alpha_{j,i}H_{i,l}=\alpha_{j,i}\alpha_{l,i}f_{in}({u}_{hi}^{-}(b_i))$. Consult Figure  \ref{Vacek_Figure_Distribution_error}(b) in the following.

\textbf{Case 1.} If ${u}_{hi}^{-}(b_i)\ge u^*$ then $f_{in}({u}_{hi}^{-}(b_i)) = f(u^*)$ and $f_{out}({u}_{hj}^{+}(a_j)) \ge f(\tilde{u}_{j})=\alpha_{j,i}f(u^*)$. Hence $H_{i,j}=\alpha_{j,i}f(u^*)$ and similarly $H_{i,l}=\alpha_{l,i}f(u^*)$. Therefore $\alpha_{l,i}H_{i,j}=\alpha_{j,i}H_{i,l}$.

\textbf{Case 2.}  If ${u}_{hi}^{-}(b_i)< u^*$ then $f_{in}({u}_{hi}^{-}(b_i)) = f({u}_{hi}^{-}(b_i))$ and $f_{out}({u}_{hj}^{+}(a_j)) \ge f(\tilde{u}_{j})=\alpha_{j,i}f({u}_{hi}^{-}(b_i))$. Hence $H_{i,j}=\alpha_{j,i}f({u}_{hi}^{-}(b_i))$ and similarly $H_{i,l}=\alpha_{l,i}f({u}_{hi}^{-}(b_i))$. Therefore $\alpha_{l,i}H_{i,j}=\alpha_{j,i}H_{i,l}$.
\end{proof}

We note that Theorems \ref{theorem_distribution2a_out} and \ref{theorem_distribution2a_in} are more general in that they allow larger densities on the outgoing roads than Theorem \ref{theorem_distribution2_in}, while still satisfying the traffic distribution condition (\ref{traffic_distribution_exact}) exactly. Theorem \ref{theorem_distribution2_in} requires  ${u}_{hj}^{+}(a_j)\leq u_*$ on all outgoing roads. However, Theorems \ref{theorem_distribution2a_out} and \ref{theorem_distribution2a_in} allow a weaker condition ${u}_{hj}^{+}(a_j)\leq \tilde{u}_{j}$ for some $\tilde{u}_{j}>u^*$ provided that ${u}_{hi}^{-}(b_i)< u^*$, i.e. if the incoming roads are not too full. Moreover, the value of this allowed density $\tilde{u}_{j}$ is higher for the $\alpha$--inside Godunov than for $\alpha$--outside, as can be seen when comparing the values of $\tilde{u}$ in Figures \ref{Vacek_Figure_Distribution_error}(a) and \ref{Vacek_Figure_Distribution_error}(b). Thus the $\alpha$--inside numerical flux allows for a larger traffic flow through the junction, while still satisfying the traffic distribution condition (\ref{traffic_distribution_exact}).

\section{Numerical results}\label{sec_results}
Here we present numerical experiments comparing the three Godunov-like numerical fluxes considered in this paper. As for the implementation, integrals over individual elements in (\ref{DG_Weak}) are evaluated using Gaussian quadrature rules. Basis functions of the space $S_h$ are taken as Legendre polynomials on individual elements, where the support of each basis function is a single element. By writing $u_h$ in terms of basis functions in space and setting the test function $\varphi$ to elements of the basis, equation (\ref{DG_Weak}) reduces to a system of ordinary differential equations which is solved by Adams–Bashforth methods. 

The DG method is much less susceptible to the Gibbs phenomenon than the finite element method, however spurious oscillations can still occur locally in the vicinity of discontinuities or steep gradients in the solution. There are several approaches how to treat these local oscillations, e.g. adding local artificial diffusion. In our case, we apply limiters to the DG solution. In our implementation, we use the \emph{modified minmod limiter} from  \cite{Stabilization_original}, cf. also \cite{Stabilization}.

Often the solution of (\ref{DG_Hyperbolic_problem}) is a physical quantity which satisfies some admissibility conditions, e.g. the physical density must be positive. If we obtain a solution which is not in the admissible interval, e.g. due to overshoots or undershoots, the problem can become ill-posed or even undefined. This is our case, since the traffic density ${u}$ must naturally satisfy ${u}\in[0,{u}_{\max}]$. The DG method by itself does not guaranty such bounds are satisfied for the discrete solution. Limiters usually prevent this from happening, however in traffic flows, it is natural that entire regions of the computational domain have ${u}=0$ or ${u}={u}_{\max}$ and it is easy for the algorithm to produce e.g. negative density due to round-off errors. To prevent this from happening, we use the following procedure. If the average density on an element $K$ is in the admissible interval, we decrease the slope of our solution so that the modified density lies in $[{u}_{\min},{u}_{\max}]$ similarly as in the limiting procedure. The important property is that the element average does not change after the application of the limiter. As further insurance, if the average density on an element $K$ is not in the admissible interval $[0,{u}_{\max}]$, then we change the solution such that ${u}\equiv 0$ or ${u}\equiv{u}_{\max}$ on the whole element $K$. The latter case, when the average density on an element is not in the admissible interval is extremely rare and, for us, serves as an indicator that the time step is too large or the mesh is too coarse. Since in this case the described procedure does not conserve the total number of vehicles, we rather decrease the time step or increase the number of elements. For polynomials of higher degree, we can use the method described in the paper \cite{High_order_limiters} by Zhang, Xia and Shu, which reduces to the procedure described above in the simplest case of piecewise linear approximations in 1D.

We consider a simple network with one incoming road (Road 1, colored red in Figures \ref{Vacek_Figure_Numerical_experiment_1} and \ref{Vacek_Figure_Numerical_experiment_2}) and two outgoing roads, Road 2 (green) and Road 3 (blue). The network will be closed at their endpoints ($a_1$, $b_2$ and $b_3$) meaning that the inflow density at $a_1$ is equal to zero and the outflow densities at $b_2$ and $b_3$ are maximal (i.e. equal to 1). Thus, we can check the total number of cars, because we have neither inflow nor outflow into the network. We choose $\alpha_{2,1}=0.75$ and $\alpha_{3,1}=0.25$. The length of all roads is $1$. We use the combination of the explicit Euler method (step size $\tau=10^{-4}$) and DG method (number of elements $N=150$ on each road). We calculate the piecewise linear approximations of solutions and we use two Gaussian quadrature points in each element. We use Greenshields model with $v_{\max}=1$ and ${\rho}_{\max}=1$.

In the following section, we set the initial condition to see the differences between approaches. We choose congested examples to demonstrate the distribution error from Lemma \ref{theorem_distribution_in}. In non--congested cases, the traffic distribution error is zero.

\subsection{$\alpha$-outside vs. $\alpha$-inside Godunov flux}
First, we compare the $\alpha$-outside and $\alpha$-inside Godunov fluxes. This example shows that the traffic flow from the incoming road in the $\alpha$-outside case isn't as high as in the $\alpha$-inside case. We also expect a distribution error in $\alpha$-outside, which corresponds to Lemma \ref{theorem_distribution_out}. We use the initial conditions
\begin{align*}
{u}_{0,1}(x)=
\begin{cases}
0.5,\\
0.5,\\
\end{cases}
\quad
{u}_{0,2}(x)=
\begin{cases}
0.75,\\
0,\\
\end{cases}
\quad
{u}_{0,3}(x)=
\begin{cases}
0.25,& x\in[1,1.5],\\
0,& x\in(1.5,2],\\
\end{cases}
\end{align*}
cf. Fig. \ref{Vacek_Figure_Numerical_experiment_1_a}. The total amount of vehicles is $0.5$ on Road 1. These cars are distributed into Road 2 (it has $0.375$ cars already) and Road 3 (it has $0.125$ cars already) by the distribution coefficients. At the end, we can expect $0.75$ cars on Road 2 and $0.25$ cars on Road 3. 

\begin{figure}[t!]\centering
\subfloat[$t=0$.]{\label{Vacek_Figure_Numerical_experiment_1_a}\includegraphics[width=4.3in]{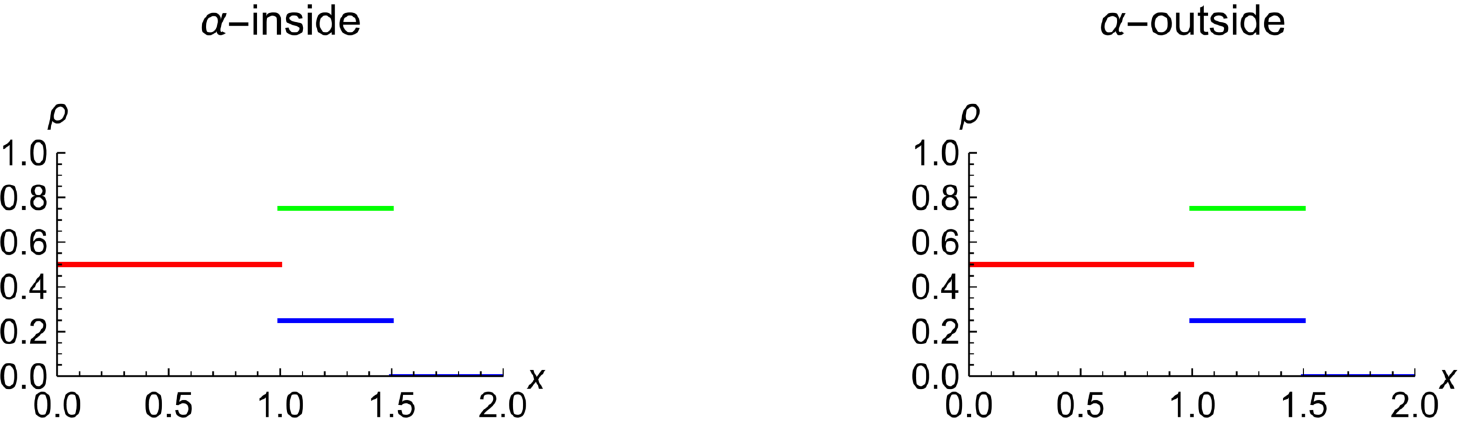}}
\hspace{0pt}
\subfloat[$t=0.6$.]{\label{Vacek_Figure_Numerical_experiment_1_b}\includegraphics[width=4.3in]{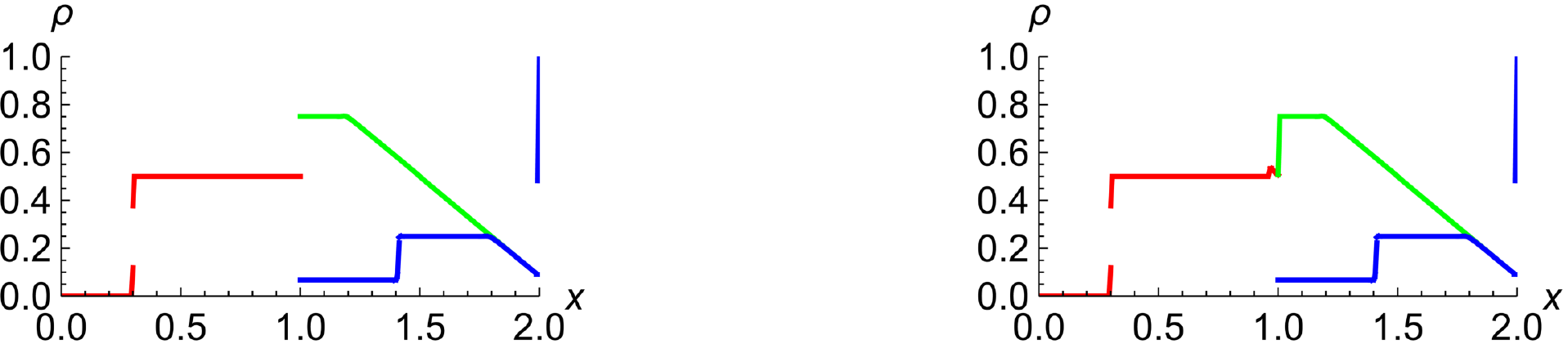}}
\hspace{0pt}
\subfloat[$t=1.2$.]{\label{Vacek_Figure_Numerical_experiment_1_c}\includegraphics[width=4.3in]{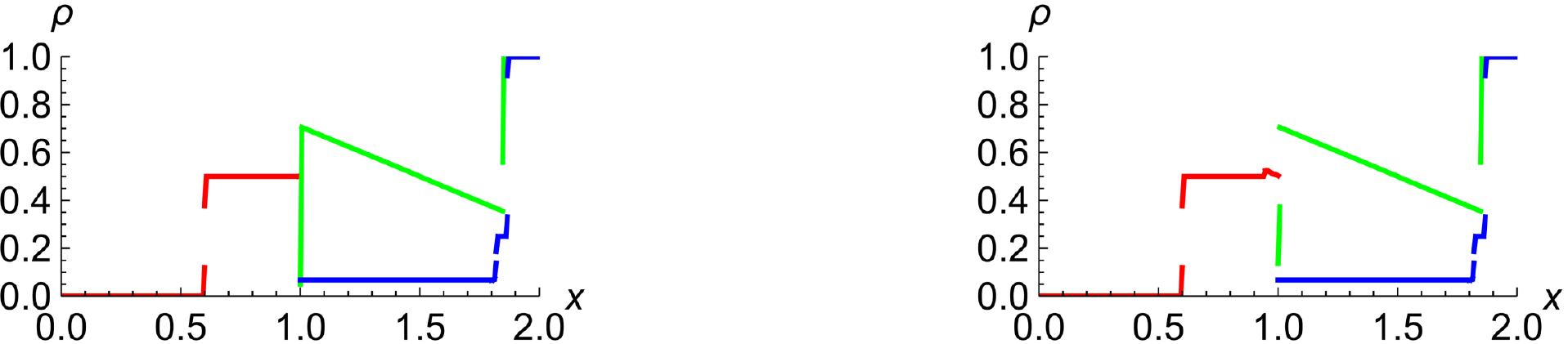}}
\hspace{0pt}
\subfloat[$t=1.8$.]{\label{Vacek_Figure_Numerical_experiment_1_d}\includegraphics[width=4.3in]{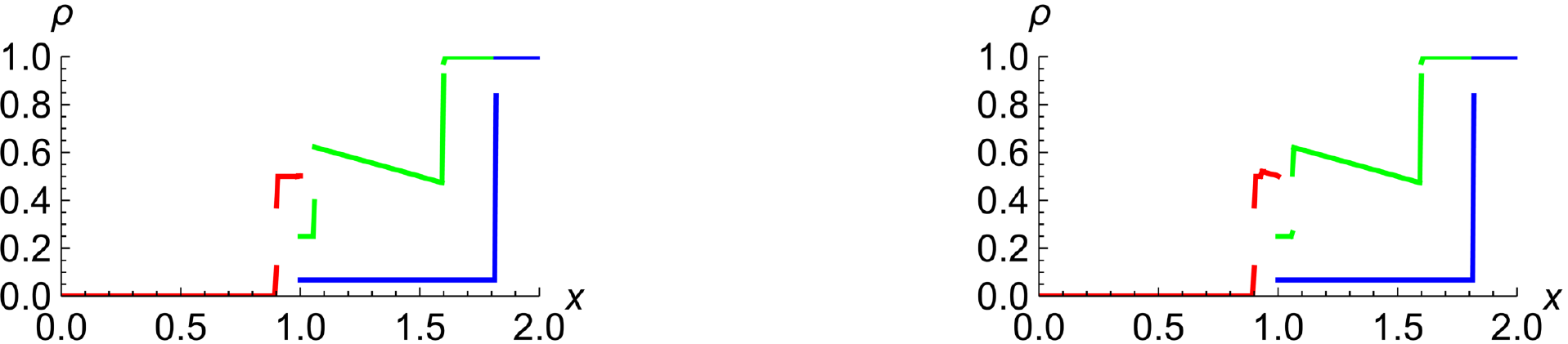}}
\hspace{0pt}
\subfloat[$t=2.4$.]{\label{Vacek_Figure_Numerical_experiment_1_e}\includegraphics[width=4.3in]{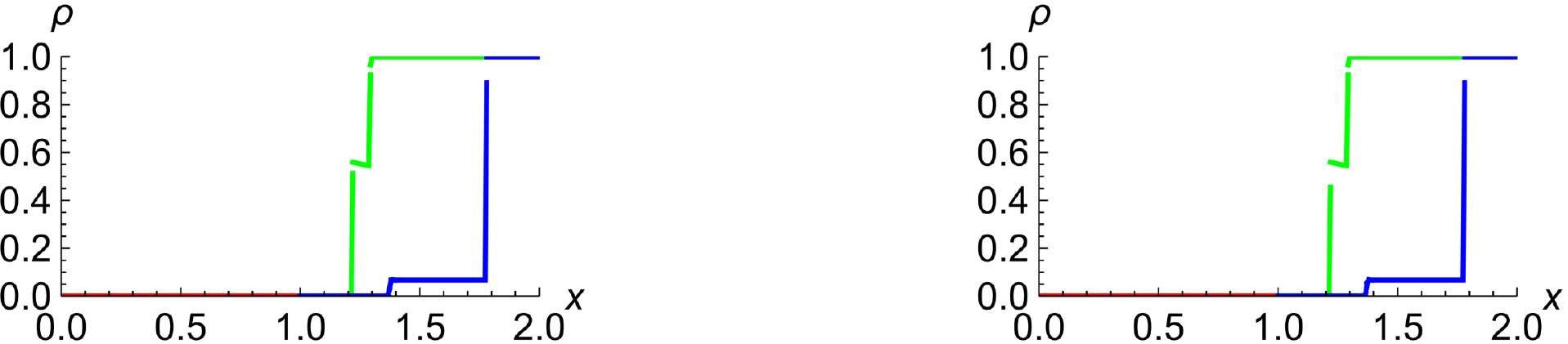}}
\hspace{0pt}
\subfloat[$t=3$.]{\label{Vacek_Figure_Numerical_experiment_1_f}\includegraphics[width=4.3in]{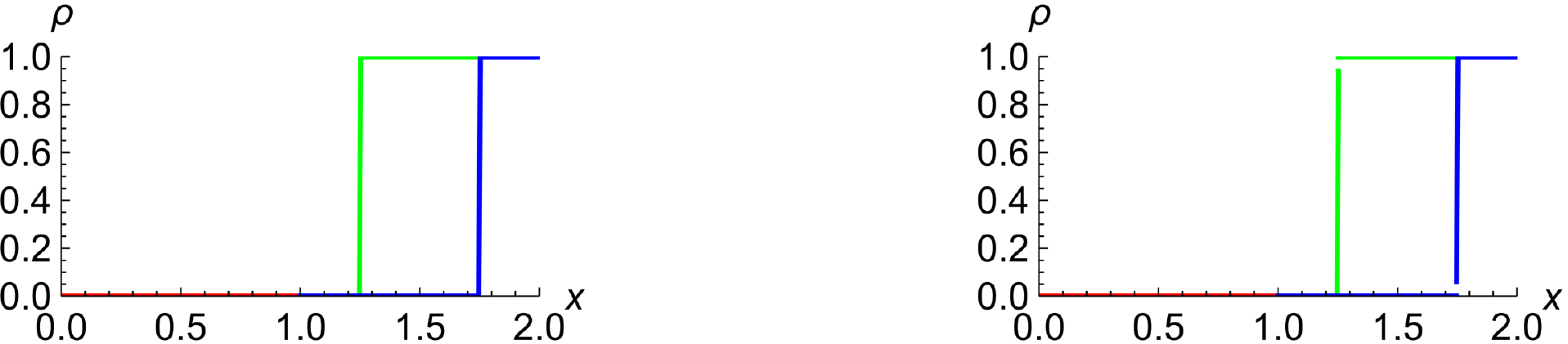}}
\caption{Comparison of $\alpha$-inside (left) and $\alpha$-outside (right) Godunov flux on network with \textcolor{red}{Road 1}, \textcolor{green}{Road 2} and \textcolor{blue}{Road 3}.}
\label{Vacek_Figure_Numerical_experiment_1}
\end{figure}

We can see the results in Fig. \ref{Vacek_Figure_Numerical_experiment_1}. We can observe, that the $\alpha$-outside Godunov numerical flux creates a traffic congestion at the end of Road 1, see Figure \ref{Vacek_Figure_Numerical_experiment_1_b} and \ref{Vacek_Figure_Numerical_experiment_1_c}. This effect is very subtle, but important -- in Figures \ref{Vacek_Figure_Numerical_experiment_1_b} and \ref{Vacek_Figure_Numerical_experiment_1_c}, for the $\alpha$-outside flux there is a slight increase in density at the end of Road 1 immediately before the bifurcation to Roads 2 and 3. The key point is that the density there locally rises above $u^*$, which is, in a certain sense, the definition of a congestion. This effect does not occur for $\alpha$-inside (the density remains beneath $u^*$) and is artificial in the $\alpha$-outside case, as this congestion should not occur -- Roads 2 and 3 are able to take in even the larger inflow which the $\alpha$-inside Godunov flux prescribes there. However, in this case the $\alpha$-outside gives a smaller inflow to Roads 2 and 3, leading to the congestion at the end of Road 1. This is due to the position of the distribution coefficient and evaluation of the minimum in the calculation of the numerical flux between Road 1 and Road 2. Because $f_{out}^{(2)}<f_{in}^{(1)}$ and $\alpha_{2,1} f_{in}^{(1)}<f_{out}^{(2)}$, the $\alpha$-inside case takes $\alpha_{2,1} f_{in}^{(1)}$ as the minimizer in the flux definition, whereas the $\alpha$-outside case takes $f_{out}^{(2)}$ and multiplies it by $\alpha_{2,1}$. Thus, $\alpha$-outside has lower inflow from Road 1 and makes a congestion. Once the traffic density on Road 2 decreases, the inflow is same in both cases, cf. Figure \ref{Vacek_Figure_Numerical_experiment_1_d} and \ref{Vacek_Figure_Numerical_experiment_1_e}.

The final results are in Fig. \ref{Vacek_Figure_Numerical_experiment_1_f}. The numerical flux with $\alpha$--inside has $0.75$ cars on Road 2 and $0.25$ on Road 3, i.e. there is no distribution error and traffic is distributed exactly according to the drivers' preferences. The numerical flux with $\alpha$--outside gives approx. $0.7498$ cars on Road 2 and $0.2502$ cars on Road 3. In this case, we have a small distribution error which is caused by a traffic congestion between Road 1 and Road 2, and therefore some drivers prefer Road 3 instead of their original preference of Road 2. Note that this congestion is caused by the choice of the $\alpha$--outside Godunov flux, not by the traffic situation in itself.

\subsection{$\alpha$-inside Godunov flux vs. Maximum possible traffic flow}
Second, we compare the $\alpha$-inside Godunov flux and Maximum possible traffic flow flux. We use initial conditions
\begin{align*}
{u}_{0,1}(x)=
\begin{cases}
0,\\
1,\\
\end{cases}
\quad
{u}_{0,2}(x)=
\begin{cases}
1,\\
0,\\
\end{cases}
\quad
{u}_{0,3}(x)=
\begin{cases}
0,&\qquad x\in[0,0.5],\\
0,&\qquad x\in(0.5,1],\\
\end{cases}
\end{align*}
cf. Fig. \ref{Vacek_Figure_Numerical_experiment_2_a}. The total amount of vehicles is $0.5$ cars on Road 1. These cars are distributed into Road 2 (it has $0.5$ cars already) and Road 3 according to the distribution coefficients. At the end, we can expect $0.875$ cars on Road 2 and $0.125$ cars on Road 3.

\begin{figure}[t!]\centering
\subfloat[$t=0$.]{\label{Vacek_Figure_Numerical_experiment_2_a}\includegraphics[width=4.3in]{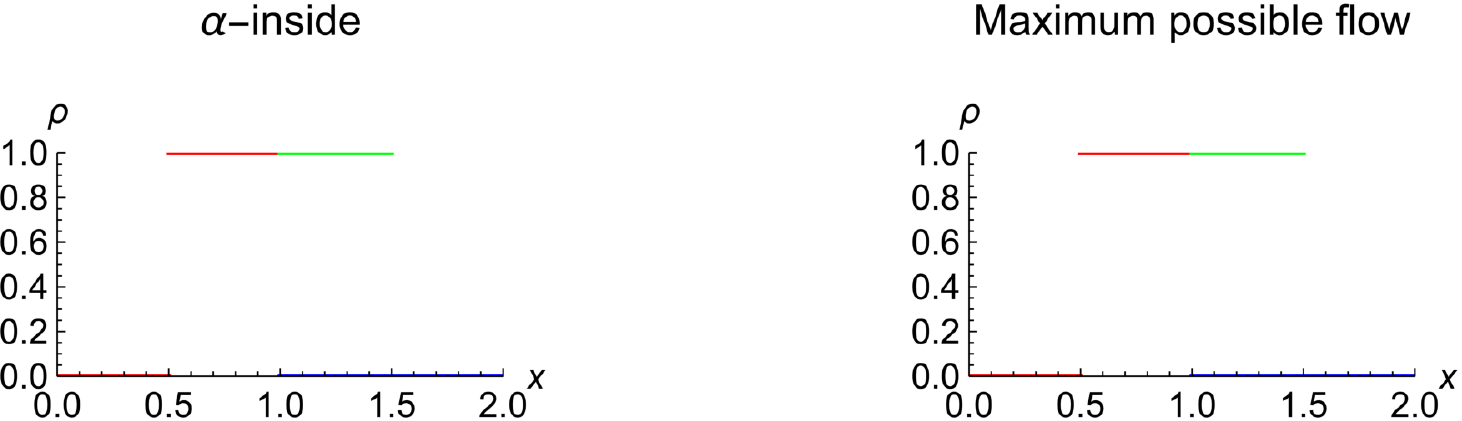}}
\hspace{0pt}
\subfloat[$t=0.25$.]{\label{Vacek_Figure_Numerical_experiment_2_b}\includegraphics[width=4.3in]{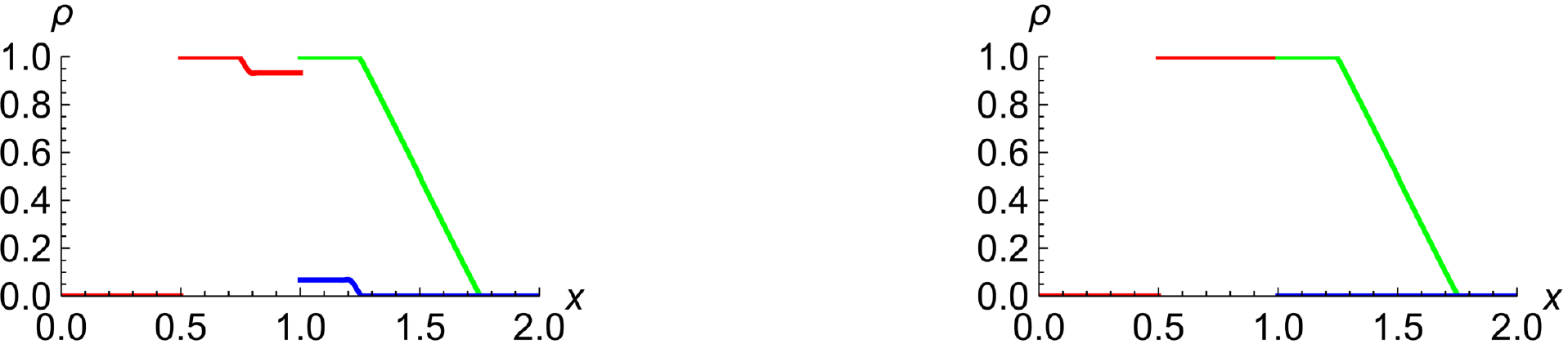}}
\hspace{0pt}
\subfloat[$t=0.5$.]{\label{Vacek_Figure_Numerical_experiment_2_c}\includegraphics[width=4.3in]{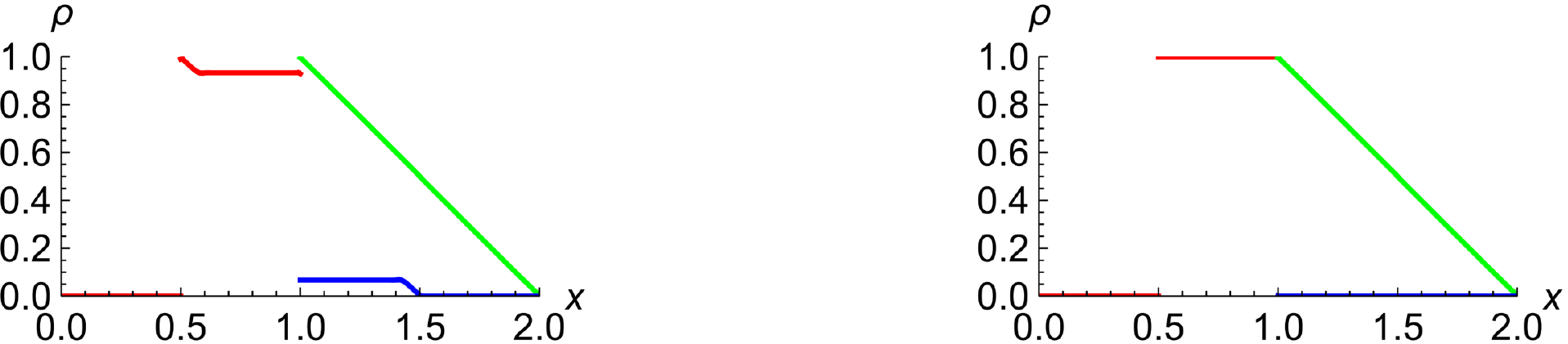}}
\hspace{0pt}
\subfloat[$t=1.25$.]{\label{Vacek_Figure_Numerical_experiment_2_d}\includegraphics[width=4.3in]{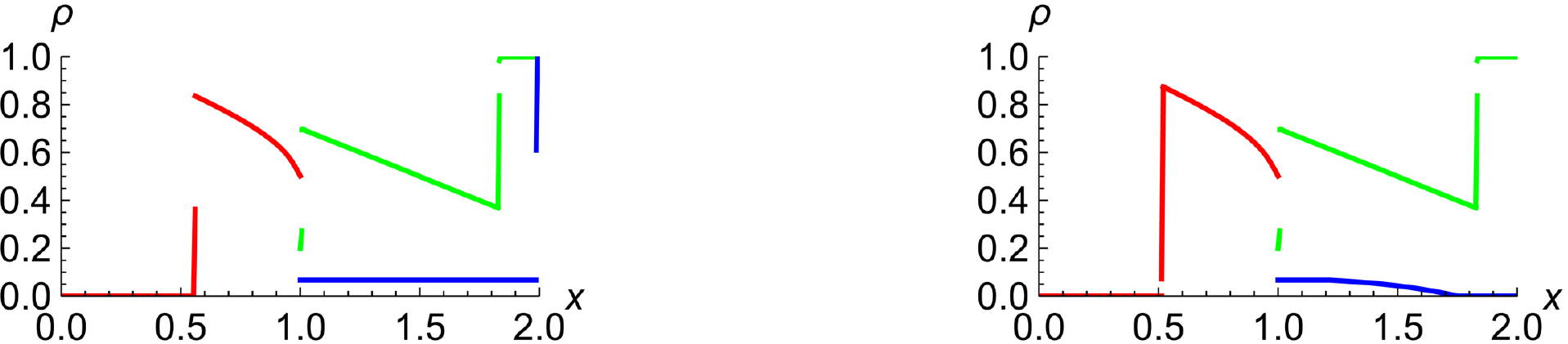}}
\hspace{0pt}
\subfloat[$t=2.5$.]{\label{Vacek_Figure_Numerical_experiment_2_e}\includegraphics[width=4.3in]{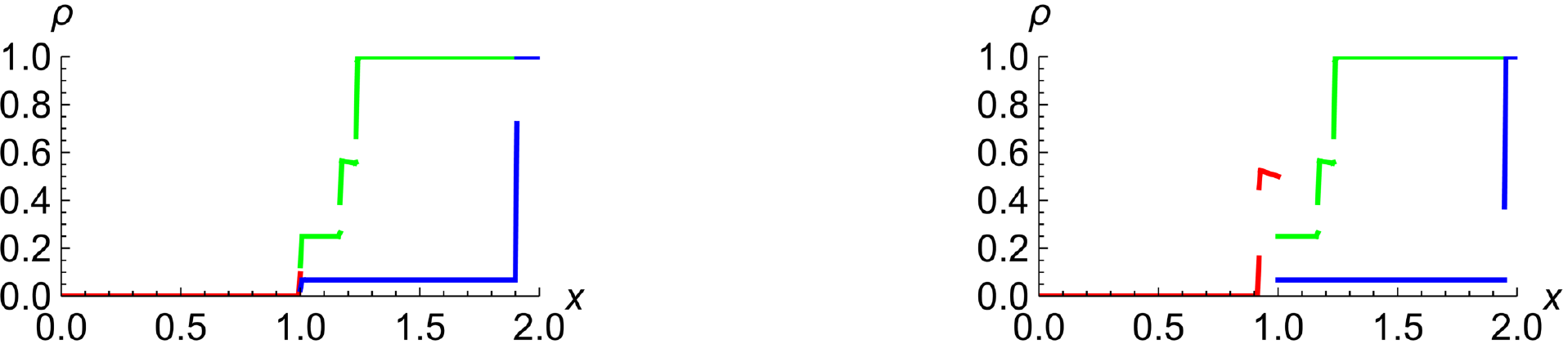}}
\hspace{0pt}
\subfloat[$t=4$.]{\label{Vacek_Figure_Numerical_experiment_2_f}\includegraphics[width=4.3in]{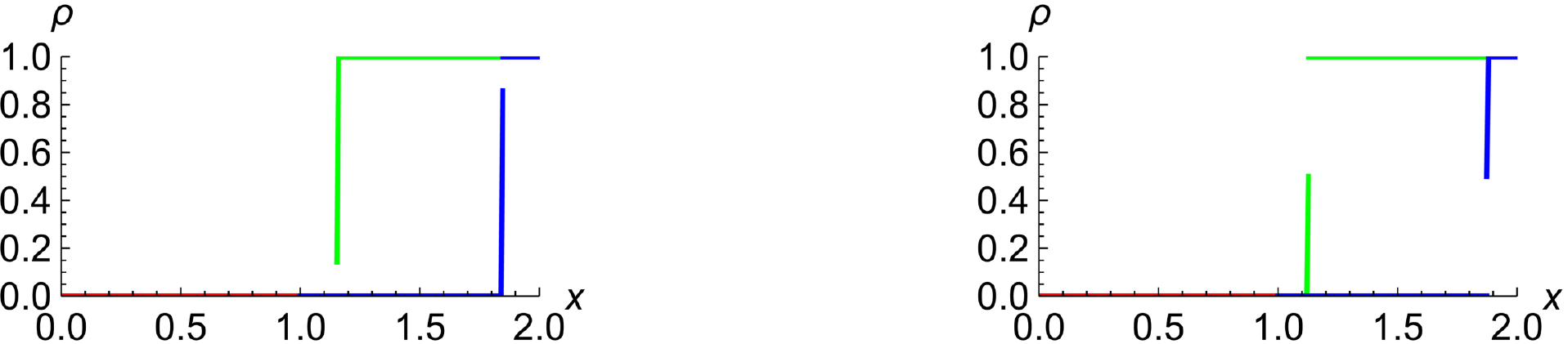}}
\caption{Comparison of $\alpha$-inside Godunov flux (left) and Maximum possible traffic flow (right) on network with \textcolor{red}{Road 1}, \textcolor{green}{Road 2} and \textcolor{blue}{Road 3}.}
\label{Vacek_Figure_Numerical_experiment_2}
\end{figure}

We can see the results in Fig. \ref{Vacek_Figure_Numerical_experiment_2}. If we compare the flow through the junction in Fig. \ref{Vacek_Figure_Numerical_experiment_2_b} and \ref{Vacek_Figure_Numerical_experiment_2_c}, we can see that our numerical flux allows flow between Road 1 and Road 3 while the Maximum possible flow doesn't. From the time $t=0.5$, the congestion on Road 2 decreases and the traffic flows from Road 1 to Road 2 and 3. The inflow from Road 1 in the case with maximum possible flow is still lower than in case with our numerical flux. In Figure \ref{Vacek_Figure_Numerical_experiment_2_d}, the inflow is the same in both cases. At $t=2.5$, Road 1 is almost empty in the case of our Godunov flux (approximately $0.0003$ cars), whereas in the case of Maximum possible flow there are still approx. $0.0414$ cars, cf. Figure \ref{Vacek_Figure_Numerical_experiment_2_e}. It is obvious, that the movement of all cars finished earlier in the case of the Godunov flux.

The final results are in Fig. \ref{Vacek_Figure_Numerical_experiment_2_f}. The maximum possible traffic flow has $0.875$ cars on Road 2 and $0.125$ on Road 3, i.e. there is no distribution error. The Godunov numerical flux with $\alpha$--inside has approx. $0.8438$ cars on Road 2 and approx. $0.1562$ cars on Road 3. In this case, we have a non--zero distribution error which is caused by a traffic jam on Road 2, and therefore some drivers prefer rather Road 3.

We note that for both of the compared fluxes, the solution on Road 2 (green) is identical in Figures \ref{Vacek_Figure_Numerical_experiment_2_a}--\ref{Vacek_Figure_Numerical_experiment_2_d}. The difference is in the distribution of cars on Road 1 and Road 3 (blue), where the maximum possible flux gives a zero flow due to the traffic jam, unlike the $\alpha$-inside flux which allows a small number of cars to enter Road 3 due to the traffic distribution error. The results on Road 2 start to differ in Figure \ref{Vacek_Figure_Numerical_experiment_2_e}, since then all cars are evacuated from Road 1 (here the nonzero flux to Road 3 helped `drain' Road 1 more quickly) in the $\alpha$-inside Godunov case, unlike the maximum possible flux where there are still some cars left on Road 1 supplying an inflow to Roads 2 and 3.

As can be seen from these two examples, the effect of the traffic distribution errors is rather subtle, but in our view leads to more realistic results, where human drivers tend to adapt to current traffic situations and change their original preferences in the presence of traffic jams, e.g. by taking alternative routes. In the approach used in the maximum possible traffic flow, the drivers strictly adhere to their original preferences under all circumstances, even in extreme situations when one of the outgoing roads is completely jammed and the other is empty. We propose two interpretations of these phenomena. The first one is that the more flexible Godunov flux describes human drivers which adapt to current situations, while the Maximum flow describes e.g. a fleet of communicating autonomous vehicles, which optimize (maximize) the total flow through the junction, while strictly adhering to the predetermined routes. On the other hand, a typical human driver does not care about maximizing the total flux through the junction, he simply wants to get through the junction to his desired outgoing road and does not really care what happens on the other roads. 

Another possible interpretation is the presence of dedicated turning lanes in front of the junction (Godunov) and their absence (maximum flow).  If dedicated turning lanes are not present, a traffic jam on one of the outgoing roads causes a congestion in the whole junction, as cars which want to go to another possibly empty road cannot do so, since they cannot overtake the standing vehicles. It is the presence of turning lanes that allows these cars to pass the other standing cars, resulting in a nonzero flow through the junction and a small violation of the predetermined traffic distribution coefficients. Once again we remind that if the roads are sufficiently free (in the sense of Section \ref{Section_Distribution_Error}), the drivers strictly adhere to their original preferences even in the case of the Godunov flux. It is only in the presence of congestions that the flexibility of the Godunov approach manifests itself (in the form of the traffic distribution error).

\section{Conclusion}
In this paper, we dealt with the construction and analysis of two new numerical fluxes for traffic flows on networks. We use the discontinuous Galerkin method to discretize the governing equations in the form of first order nonlinear hyperbolic conservation laws describing the traffic flow on individual roads. The main contribution is the construction of two new numerical fluxes at the network junctions that are a generalization of the Godunov numerical flux. The construction is an extension of our previous work \cite{Prvni_clanek} which was based on a generic numerical flux, rather than Godunov. We prove basic properties of the two newly proposed Godunov-like fluxes, namely the conservativity of the resulting numerical method (via discrete Rankine-Hugoniot conditions) and analyze situations when the predetermined drivers' preferences are satisfied or possibly violated. Specifically, the drivers' preferences at junctions are given by traffic distribution coefficients. We show that if the junction is not congested, the traffic flows according to these preferences. Once the junction becomes congested, there can be a small traffic distribution error which we interpret either as factoring of human behavior into the model or the existence of dedicated turning lanes in front of the junction, as opposed to single-lane roads. We demonstrate these phenomena numerically and compare with the approach to the construction of numerical fluxes taken in \cite{RKDG}. One of the advantages of our Godunov-like numerical fluxes is the simplicity of their explicit construction for all types of junctions, unlike the approach of \cite{RKDG} and \cite{Networks}, which requires the solution of a linear programming problem. In subsequent papers we will prove an entropy inequality for the scheme along with $L^2$ stability and error estimates and analyze the behavior of limiters at junctions.

\bibliographystyle{spmpsci}      
\bibliography{Vacek_bibliography}   

\end{document}